\documentclass[a4paper,report, 11pt]{article}
\usepackage{amsmath,a4wide}
\usepackage{amssymb}
\usepackage{amscd}
\usepackage{epsfig}
\usepackage{graphics}
\usepackage{multirow}

\usepackage{graphicx}

\usepackage{subfigure} 
\usepackage{epstopdf}
\usepackage{booktabs}

\usepackage{ifpdf}
\usepackage[a4paper,top=1.5in, bottom=1.5in, left=0.5in, right=0.5in]{geometry}
\usepackage{changes}
\usepackage[makeroom]{cancel}
\usepackage{bbm}

\usepackage{amsthm}
\usepackage{color}
\definecolor{amaranth}{rgb}{0.9, 0.17, 0.31}	
\definecolor{myblue}{rgb}{0.2,0,0.9}
\definecolor{auburn}{rgb}{0.43, 0.21, 0.1}
\definecolor{bittersweet}{rgb}{1.0, 0.44, 0.37}
\definecolor{blue-violet}{rgb}{0.54, 0.17, 0.89}
\usepackage[authoryear]{natbib}

\usepackage{hyperref}

\hypersetup{
	colorlinks,linkcolor=myblue,citecolor=amaranth,filecolor=magenta,urlcolor=myblue}

\newtheorem{pro}{Proposition}[section]
\newtheorem{pr}{Problem}

\newtheorem{thm}{Theorem}[section]

\newtheorem{lem}{Lemma}[section]

\newtheorem{as}{Assumption}[section]
\newtheorem{rem}{Remark}[section]

\def\d{\delta}
\def\D{\Delta}
\def\e{\epsilon}

\def\m{\mu}

\def\s{\sigma}
\def\t{\theta}

\def\k{\kappa}

\begin{document}
	
	\title{\LARGE {\bf A problem of optimal switching and singular control with discretionary stopping in portfolio selection \footnote{ Junkee Jeon is supported by the National Research Foundation of Korea (NRF) grant funded by the Korea government [Grant No. NRF-2020R1C1C1A01007313]. Hyeng Keun Koo is supported by NRF grant [Grat No. NRF-2020R1A2C1A01006134].
	 }}}
	
	\author{
		Junkee Jeon \footnote{E-mail: {\tt junkeejeon@khu.ac.kr}\;Department of Applied Mathematics, Kyung Hee University, Korea.} 
		\and
		Hyeng Keun Koo\footnote{E-mail: {\tt hkoo@ajou.ac.kr}\;Department of Financial Engineering, Ajou University, Korea.}
	}
	
	\date{\today}
	
	\maketitle \pagestyle{plain} \pagenumbering{arabic}
	
                    	\abstract{In this paper we study the optimization problem of an economic agent who chooses a job and the time of retirement as well as consumption and portfolio of assets. The agent is constrained in the ability to borrow against future income. We transform the problem into a dual two-person zero-sum game, which involves a controller, who is a minimizer and chooses a non-increasing process, and a stopper, who is a maximizer and chooses a stopping time.
                        We derive the Hamilton-Jacobi-
                    	Bellman quasi-variational inequality(HJBQV) of a max-min type arising from the game. We provide a solution to the HJBQV and verification that it is the value of the game. We establish a duality result which allows to derive the optimal strategies and value function of the primal problem  from those of the dual problem. }
	
	\vspace{1.0cm} {\em Keywords} : consumption, portfolio selection, job switch, early retirement, borrowing constraint, zero-sum game, Nash-equilibrium, Hamilton-Jacobi-Bellman quasi-variational inequality \\

\newpage

\section{Introduction}

In this paper we study the optimization problem of an economic agent who chooses a job and the time of retirement as well as consumption and portfolio of assets.
There has been increasing interest in studying life-cycle patterns of consumption and asset allocation. An important realistic aspect in the study is to incorporate labor supply and human capital and their effects on consumption and risky investments. Job choice and retirement decisions are two important factors determining labor supply and human capital. 

As the average life span and flexibility in the job market increase, there exist higher chances that people change their jobs, being motivated by the consideration of leisure and job satisfaction than that of higher salaries (\citet{SKS}). Thus, a job which promises more leisure time but with a lower salary can be a substitute for retirement. Taking such a job allows one to have an option to choose another job with a higher salary and less leisure time and thus to keep a positive value of human capital, while retirement is an irreversible decision and makes human capital  equal to 0 permanently. We study the choice between the two substitutes, a job with more leisure time and lower salary and permanent retirement in this paper. 

Another important aspect regarding human capital is the constraint which restricts people from  borrowing against future income.  
Researchers have shown that the borrowing constraint has significant effects on consumption, investment, labor supply, and wealth accumulation (\citet{HP}, \citet{DF2006}, \citet{Rendon2006}, \citet{DL2010}). We study the optimization problem of a borrowing constrained agent. 

Mathematically, the problem is a combination of continuous control, discrete control and optimal stopping. The optimal retirement problem alone introduces complication, as retirement decision interacts with optimal consumption and portfolio choice (\citet{YK}). The borrowing constraint significantly complicates the problem, as it makes the financial market incomplete from the agent's perspective (\citet{DL2011}).

	We extend the previous work on optimal consumption and portfolio choice (\citet{FP2007}, \citet{DL2010, DL2011}, and \citet{Lee-et-al-2019}) by considering the job choice or retirement decision. Moreover, we employ a class of general utility functions.  \citet{FP2007} and \citet{DL2010, DL2011} consider the optimization problem with a constant relative risk aversion(CRRA) utility function. Thanks to the homogeneity of the CRRA utility function, the free boundaries associated with retirement and borrowing constraints are characterized by only one algebraic equation in their work. In contrast, those in our problem are expressed as solutions of two coupled algebraic equations, which makes the problem more complex and difficult.

 There is a substantial work investigating the combined problem of the  optimal retirement decision and the borrowing constraint (\citet{FP2007}, \citet{CSS2007}, \citet{DL2010, DL2011}, \citet{LS2011}). While these studies utilize the dual method,  the dual problem is not  specified in a clear and comprehensive manner,  and  often lack  mathematical rigor for the verification of the dual problem and the duality theorem. To clarify the issue, we take a novel approach by transforming the problem into a two-person zero-sum game, which involves a controller, who is a minimizer and chooses a non-increasing process, and a stopper, who is a maximizer and chooses a stopping time.

Specifically, the transformation involves two critical steps. First, we transform the wealth dynamics into a budget constraint in static form. We use the martingale approach developed by \citet{KLS} and \citet{CoxH}. In order to incorporate the borrowing constraint we introduce a non-increasing shadow  price process as in \citet{HP} and \citet{KarouiJeanblanc1998}. Next we adopt the approach to change the order of optimization with in \citet{KW2000} and consider the problem of finding the optimal retirement time after selecting the optimal consumption in a dual problem. After the two steps are carried out, the problem becomes that involving  the choice of an optimal retirement time which maximizes a dual objective function and that of a shadow price process which minimizes the dual objective function. Hence, the problem can be cast into a two-person zero-sum game; one player maximizes the objective by choosing the retirement time and the other player minimizes the objective by choosing the shadow price process. To our knowledge, this paper is the first that clarified the dual problem in the optimal retirement decision of a borrowing constrained agent.

We derive a Hamilton-Jacobi-Bellman quasi-variational inequality (HJBQV) which is satisfied by the value function of the game. The {HJBQV} involves both minimum and maximum. We construct {an explicit-form} solution of the HJBQV and provide a verification that the solution is the value function of the game. Finally we establish a duality theorem which allows to derive the value function of the original problem from that of the game. By the duality theorem we can derive the optimal strategies of the primal problem from those of the dual problem. 

The job choice is made by considering the trade-off between job satisfaction and salary. The agent chooses a job whenever the utility value from one job is greater than that from the other job, where the utility value includes that from the wage income. 

The dual objective function has two components: (i) the present value of the convex conjugate of optimal consumption, and (ii) the present value of the difference between the utility values after and before retirement. We show that the retirement decision is made when the  difference at the time is negative, i.e., the instantaneous utility after retirement is smaller than that before retirement. This is because retirement is an irreversible decision similar to  exercise of a financial option; the exercise of an option occurs when it is in the money (\citet{DixitPindyck}). The characteristic of the problem as a game, however, implies that the decision is influenced by the presence of the borrowing constraint. 

 The literature on the two person zero-sum game of a controller and stopper includes \citet{MS96},  \citet{W06}, \citet{KZ08}, \citet{H06}, \citet{BY11}, \citet{HSZ15}, \citet{HY15}, and references therein. In particular, the mathematical structure in our dual problem is similar to that in \citet{HSZ15} and \citet{HY15}, as the two parties are a singular controller and a discretionary stopper. We add to the literature by considering a problem of optimal switching and singular control with discretionary stopping in portfolio selection.


The paper is organized as follows. We explain the model and the optimization problem in Section \ref{sec:model}. We explain a dual problem by formulating a Lagrangian and the transformation into a two-person zero-sum game in Section \ref{sec:optimization}. We derive a HJBQV and provide optimal strategies for the optimization problem in Section \ref{sec:optimal strategies}. We conclude in Section \ref{sec:conclusion}.


\section{Model}\label{sec:model}

We consider an agent whose objective is to maximize the following expected utility: 
\begin{equation}\label{eq:expected-utility}
U\equiv \mathbb{E}\left[\int_0^{\tau}e^{-\d t}\left(u(\k_1 c_t){\bf 1}_{\{\zeta_t=\mathfrak{B}_1 \}}+u(\k_2 c_t){\bf 1}_{\{\zeta_t=\mathfrak{B}_2 \}}\right)dt +{{\bf 1}_{\{\tau<\infty \}}} \int_{\tau}^\infty e^{-\d t}u(c_t)dt\right].
\end{equation}
Here $\d>0$ is the subjective discount rate,  $c_t\ge 0$ is the  rate of consumption at time $t$, ${\bf 1}_A$ is the indicator function of set $A$,  $\zeta_t\in\{\mathfrak{B}_1, \mathfrak{B}_2 \}$ is the agent's  job  at $t$,  $\tau$ is the time of retirement,  $\k_i>0\ (i=1,2)$ is a constant, and $u(\cdot)$ is the felicity function of consumption. For simplicity, we assume that there exist two different jobs available to the agent, that is, $\zeta_t$ takes one of the two values, $\mathfrak{B}_1$ and $\mathfrak{B}_2$. Constant $\k_i$ describes the agent's satisfaction with job $\mathfrak{B}_i,$ a larger value implying higher satisfaction with the job. 

The agent receives labor income at a rate $\e_i$ when she takes job $\mathfrak{B}_i.$ 
 We assume that 
$$
0<\e_2 <\e_1\;\;\;\mbox{and}\;\;0<\k_1<\k_2<1.
$$
That is, job $\mathfrak{B}_1$ provides higher income but lower satisfaction than job $\mathfrak{B}_2.$  Here, $\k_i<1$ means that retirement provides higher satisfaction than working, and the marginal utility of consumption is greater after retirement than before retirement. 

The agent has two options: the first is to to switch jobs when working, and the second  is to retire voluntarily. After retirement, the agent cannot go back to work, and thus retirement is an  irreversible decision. In contrast, the  job switching decision is freely reversible; the agent can switch from the current job to the other job at  any time before  retirement. \\ 

\noindent {\bf Financial Market:} There exist two financial assets trading in the economy, a riskless asset and a risky asset, whose prices at $t$ are denoted by $S_{0,t}$ and $S_{i,t}$, respectively.  The asset prices satisfy the dynamics: 
$$
dS_{0,t}/S_{0,t} = {r} dt \;\;\;\mbox{and}\;\;\;dS_{1,t}/S_{1,t}= {\mu} dt + \s dW_t, 
$$
where ${r}>0$ is the risk-free rate, ${\mu}>{r}$ is the drift of the risky asset price, $\s$ is the volatility of the risky asset returns, and $W_t$ is a standard Brownian motion on  $(\Omega, \mathcal{F}_\infty, \mathbb{P}).$ We will denote the augmented filtration generated by the Brownian motion $W_t$ by $\mathcal{F}=\{\mathcal{F}_t\}_{t\ge 0}$. Here, for simplicity of the model, we assume that the investment opportunity is constant, that is, the interest rate, the mean and standard deviation of the risky asset returns are constant. Under the constant investment opportunity assumption, the assumption of two assets is without loss of generality; by the two-fund separation theorem, the general case where there exist multiple risky assets with a constant covariance matrix of  returns can be subsumed in our model by treating the market portfolio of risky assets in the general case as the single risky asset in our model (see \citet{GL}).\\

\noindent {\bf Wealth and Budget Constraint:} We will now explain the budget constraint of the agent. Let us denote the agent's  investment in the risky asset at time $t$ by $\pi_t$ (in dollar amount). The agent's wealth follows the  dynamics:
\begin{equation}\label{eq:wealth-dynamics}
dX_t = \left[rX_t + (\mu-r)\pi_t - c_t + \left(\e_1{\bf 1}_{\{\zeta_t=\mathfrak{B}_1 \}}+\e_2{\bf 1}_{\{\zeta_t=\mathfrak{B}_2 \}}\right){\bf 1}_{\{t<\tau\}} \right]dt+\s \pi_t dW_t.
\end{equation}
with $X_0=x$

Let  $\t:=(\m-r)/\s$, the risk premium on the return of the risky asset for one unit of standard deviation, or the Sharpe ratio. We define the stochastic discount factor, ${\cal H}_t$:
$$
{\cal H}_t \equiv \exp\left(-\left\{r+\dfrac{1}{2}\t^2 \right\}t-\t W_t\right).
$$
Since $\e_1>\e_2$, the natural limit to the agent's borrowing constraint is determined by the present value of the income stream:
\begin{equation}
X_t \ge -\mathbb{E}_t \left[\int_t^\infty \dfrac{{\cal H}_s}{{\cal H}_t} \e_1 ds\right]=-\dfrac{\e_1}{r},\;\;\;\mbox{for all}\;\;t\ge 0,
\end{equation}
where $\mathbb{E}_t[\cdot]=\mathbb{E}_t[\cdot\mid {\cal F}_t]$ is the  expectation conditional on the filtration ${\cal F}_t$.

We consider the \textit{ borrowing constraint} which restricts the agent from borrowing against future labor income:
\begin{equation}\label{eq:borrowing-constraint}
X_t \ge 0 \;\;\;\mbox{for all}\;\;t>0.
\end{equation}

Given $X_0=x$, we call a quadruple $(\d,c, \pi, \tau)$ admissible if 
\begin{itemize}
	\item[(a)] the job process $\zeta:=(\zeta_t)_{t=0}^\infty$ is ${\cal F}_t$-adapted and  takes one of the two values, $\mathfrak{B}_1$ and $\mathfrak{B}_2$,
	\item[(b)] $c_t \ge 0$ and $\pi_t$ are ${\cal F}_t$-progressively measurable processes satisfying the following integrability conditions:
	\begin{align*}
	\int_0^t c_s ds <\infty \;\;\mbox{a.s.},\;\;\;\mbox{and}\;\;\;\int_0^t \pi_s^2 ds<\infty\;\;\mbox{a.s.}\;\;\;\forall\;t\ge 0,
	\end{align*}
	\item[(c)] $\tau$ belongs to ${\cal S}$ which is the set of all ${\cal F}$-stopping times taking values in $[0,\infty)$,
	\item[(d)] the agent's financial wealth $X_t$ in \eqref{eq:wealth-dynamics} for all $t \ge 0$ satisfies 
	$$
	X_t \ge 0 .
	$$
\end{itemize}
We denote the set of all admissible strategies by ${\cal A}(x)$. \\

 We make the following assumptions on the felicity function to guarantee the existence of a
solution to the agent’s optimization problem:
\begin{as}\label{as:utility1}
	The felicity function $u:[0,\infty)\to \mathbb{R}$ is strictly increasing, strictly concave and continuously differentiable, and $\lim_{c\to +\infty}u'(c) = 0$. 
\end{as}

The strictly decreasing and continuous function $u':(0,\infty) \overset{\textrm{onto}}{\longrightarrow} (0,u'(0))$ has a strictly decreasing, continuous inverse $I:(0,u'(0))\overset{\textrm{onto}}{\longrightarrow}(0,\infty)$. We extend $I$ by setting $I(y)=0$ for $y\ge u'(0)$. Then, we have 
\begin{align*}
u'(I(y))=\begin{cases}
y,\;\;\;&0<y<u'(0),\\
u'(0),\;\;\;&y\ge u'(0),
\end{cases}
\end{align*}
and $I(u'(c))=c$ for $0<c<\infty$. Note that $\lim_{y\to \infty}I(y) = 0$. 

\begin{as}\label{as:utility2}
	For $y>0$, 
	\begin{align*}
	\int_0^y \eta^{-n_2}I(\eta)d\eta < \infty.
	\end{align*}
where	$n_1>0$ and $n_2<0$ are two roots of the quadratic equation:
	\begin{equation}\label{eq:quadratic}
	\dfrac{\t^2}{2}n^2 + \left(\d-r-\frac{\t^2}{2}\right)n-\d=0.\;\;\;
	\end{equation}
\end{as}

We now state the agent's optimization problem.
\begin{pr}\label{pr:main} Given $X_0=x>0$, we consider the following agent's utility maximization problem:
	\begin{equation}
	V(x) =\sup_{(c,\pi, \zeta, \tau)\in \mathcal{A}(x)}\mathbb{E}\left[\int_0^{\tau}e^{-\d t}\left(u(\k_1 c_t){\bf 1}_{\{\zeta_t=\mathfrak{B}_1 \}}+u(\k_2 c_t){\bf 1}_{\{\zeta_t=\mathfrak{B}_2 \}}\right)dt + {{\bf 1}_{\{\tau<\infty \}}}\int_{\tau}^\infty e^{-\d t}u(c_t)dt \right].
	\end{equation}
\end{pr}

\noindent Our strategy of solving Problem \ref{pr:main} consists of the following steps. 
\begin{itemize}
	\item[(Step 1)] 	By using the well-known method developed by \citet{KS} and \citet{CoxH}, we transform Problem \ref{pr:main}  into a static problem.
	\item[(Step 2)] Using the  budget constraint in static form, we formulate a Lagrangian for Problem \ref{pr:main}. By maximizing the Lagrangian, we have the candidates of optimal consumption and job. Putting these in the Lagrangian, we obtain the dual problem, which takes the form of a two-person zero-sum game of a singular stochastic controller and a discretionary stopper.
	\item[(Step 3)] Utilizing the dynamic programming principle, we derive the Hamilton-Jacobi-
	Bellman quasi-variational inequality(HJBQV) arising from the game. We provide an explicit-from solution to the HJBQV and verification that it is the value of the game.
	\item[(Step 4)] Finally, we establish the duality theorem and characterize the optimal strategies.
\end{itemize}

\section{A Dual Optimization Problem}\label{sec:optimization}

We use the martingale-dual approach  (see \citet{KLS}, \citet{CoxH}) which allows us to use the Lagrangian method to solve Problem \ref{pr:main}. In order to apply the approach we need to transform the dynamic constraint \eqref{eq:borrowing-constraint}  into a static form. \citet{HP} and \citet{KarouiJeanblanc1998} study a consumption and portfolio selection problem with constraint \eqref{eq:borrowing-constraint} by introducing a non-increasing process, which can be thought of as the integral of  infinitesimal Lagrange multipliers, but they do not consider the job-switching options and irreversible retirement decision. Recently, \citet{Lee-et-al-2019} investigate optimal job switching and consumption-investment problems under the borrowing constraint, but they also do not consider the irreversible retirement decision. Our problem described in Problem \ref{pr:main}, however, involves the job switching, borrowing constraint, and irreversible retirement decision. To obtain a constraint of static form for our problem, we combine the results of Lemma 6.3 in \citet{KW2000}, Proposition 2.2 in \citet{KarouiJeanblanc1998}, and Lemmas 1 and  2 in \citet{Lee-et-al-2019}. As a result, we derive the following proposition.

 {Let ${\cal NI}$ be the set of all non-negative, non-increasing, right-continuous processes with left limits(RCLL) and starting at $1$. }
\begin{pro}~\label{pro:static-budget}
	\begin{itemize}
		\item[(a)] For any given $\tau\in{\cal S}$, the triplet $(c,\zeta, \tau)$ of the consumption, job, and retirement strategies such that $(c,\pi, \zeta, \tau)\in{\cal A}(x)$ with a portfolio $\pi$, satisfies the following budget constraint: 
		\begin{equation}\label{eq:static1}
		\sup_{{\cal D}\in{\cal NI}}\mathbb{E}\left[\int_0^\tau {\cal H}_t {\cal D}_t\left(c_t-(\e_1{\bf 1}_{\{\zeta_t=\mathfrak{B}_1 \}} +\e_2{\bf 1}_{\{\zeta_t=\mathfrak{B}_2 \}})\right)dt+{{\bf 1}_{\{\tau<\infty \}}}{\cal H}_{\tau}X_{\tau}{\cal D}_{\tau-}\right] \le x. 
		\end{equation}
		
		Budget constraint \eqref{eq:static1} is equivalent to
		\begin{equation}\label{eq:static2}
		\sup_{\xi \in {\cal S}_{\tau}}\mathbb{E}\left[\int_0^{\xi} {\cal H}_t \left(c_t-(\e_1{\bf 1}_{\{\zeta_t=\mathfrak{B}_1 \}} +\e_2{\bf 1}_{\{\zeta_t=\mathfrak{B}_2 \}})\right)dt+{{\bf 1}_{\{\xi=\tau<\infty \}}}{\cal H}_{\tau}X_{\tau}\right] \le x,
		\end{equation}
	where ${\cal S}_{\tau}$ is the set of all stopping times such that $0\le \xi 	\le \tau$.

		\item[(b)] For any $(c,\zeta,\tau)$ satisfying budget constraint \eqref{eq:static2}, then there exists a portfolio process $\pi$ such that $(c,\pi, \zeta, \tau)\in{\cal A}(x)$.  Moreover, for $t\in[0,\tau)$, 
		\begin{equation}
			{\cal H}_tX_t = \mathbb{E}\left[\int_t^\tau {\cal H}_s \left(c_s-(\e_1{\bf 1}_{\{\zeta_s=\mathfrak{B}_1 \}} +\e_2{\bf 1}_{\{\zeta_s=\mathfrak{B}_2 \}})\right)ds +{\bf 1}_{\{\tau<\infty\}}{\cal H}_{\tau}X_{\tau}\right].
		\end{equation}
	\end{itemize}
\end{pro}

\subsection{Dual Value Function}

In this subsection we study the dual optimization problem. We first explain the agent's optimization problem after retirement. Assume  $t\ge \tau$, where $t$ is a fixed constant denoting the current time. We define the agent's optimization problem after retirement $t\ge \tau$:
\begin{equation}
V_R(X_t):= \sup_{(c,\pi)\in \mathcal{A}_R(X_t)} \mathbb{E}_t\left[\int_t^\infty e^{-\d(s-t)}u(c_s)ds\right], 
\end{equation}
where ${\cal A}_R(X_t)$ is the set of all admissible consumption and portfolio strategies for given $X_t >0$ satisfying the conditions: (i) for any $T>0$, $\int_t^T c_s ds <\infty$ a.s. and $\int_t^T \pi_s^2 ds <\infty$ a.s., (ii) for $s\ge t$, $X_s \ge 0$ with $dX_s =[rX_s+ (\m-r)\pi_s-c_s]ds+\s \pi_s dB_s$.

By the well-known results in Section 3.9 in \citet{KS}, the following duality relationship holds:
\begin{equation}
V_R(X_t) = \inf_{{\cal Y}_t>0}\left(J_R({\cal Y}_t)+{\cal Y}_t X_t\right),
\end{equation}
where ${\cal Y}_t =ye^{\d t}{\cal H}_t$ and the dual value function $J_R(y)$  after retirement is defined by 
\begin{equation}
J_R(y) := \mathbb{E}\left[\int_0^\infty e^{-\d t}\tilde{u}({\cal Y}_t)dt\right]\;\;\;\mbox{with}\;\;\;\tilde{u}(y) =\sup_{c\ge 0} \left(u(c)-yc\right].
\end{equation}

In particular, 
\begin{equation}
	V_R(X_t) =\mathbb{E}\left[\int_t^\infty e^{-\d(s-t)}u(c_s^R)ds\right],
\end{equation}
where  $c_s^R  = I({\cal Y}_s^R)$, ${\cal Y}_s^R:= {\cal Y}_t^R e^{\d(s-t)}{\cal H}_s/{\cal H}_t$ and ${\cal Y}_t^R$ is a unique solution of 
\begin{equation}
	X_t= - J_R({\cal Y}_t).
\end{equation}
The explicit form of $J_R(y)$ is given by 
\begin{footnotesize}
	\begin{equation*}
		J_R(y)= \dfrac{2}{\t^2(n_1-n_2)}\left[y^{n_2}\int_0^y \eta^{-n_2-1}\tilde{u}(\eta)d\eta +y^{n_1}\int_y^\infty \eta^{-n_1-1}\tilde{u}(\eta)d\eta\right].
	\end{equation*}
\end{footnotesize}
In Appendix \ref{sec:after-dual}, we provide useful properties of $J_R(y)$.

We formulate the Lagrangian $\mathfrak{L}$ of Problem \ref{pr:main} using constraint \eqref{eq:static1}: 
\begin{align}
\label{eq:Lagrangian1}
\mathfrak{L}\equiv&\sup_{\{c_t, \zeta_t\}}\left\{\mathbb{E}\left[\int_0^{\tau}e^{-\d t}\left(u(\k_1 c_t){\bf 1}_{\{\zeta_t=\mathfrak{B}_1 \}}+u(\k_2 c_t){\bf 1}_{\{\zeta_t=\mathfrak{B}_2 \}}\right)dt + {{\bf 1}_{\{\tau<\infty \}}}\int_{\tau}^\infty e^{-\d t}u(c_t)dt \right]\right.\\
+&\left. y\left(x-\mathbb{E}\left[\int_0^\tau {\cal H}_t{\cal D}_t \left(c_t-(\e_1{\bf 1}_{\{\zeta_t=\mathfrak{B}_1 \}} +\e_2{\bf 1}_{\{\zeta_t=\mathfrak{B}_2 \}})\right)dt+{{\bf 1}_{\{\tau<\infty \}}}{\cal H}_\tau{X}_{\tau} {\cal D}_{\tau-} \right]\right)\right\} \nonumber\\
\le &  \sup_{\{c_t, \zeta_t\}}\left\{\mathbb{E}\left[\int_0^{\tau}e^{-\d t}\left((u(\k_1 c_t)-c_t{\cal Y}_t{\cal D}_t+\e_1{\cal Y}_t{\cal D}_t){\bf 1}_{\{\zeta_t=\mathfrak{B}_1 \}}+(u(\k_2 c_t)-c_t{\cal Y}_t{\cal D}_t+\e_2{\cal Y}_t{\cal D}_t){\bf 1}_{\{\zeta_t=\mathfrak{B}_2 \}}\right)dt\right.\right.\nonumber\\ +&\left.\left. {{\bf 1}_{\{\tau<\infty \}}} e^{-\d \tau}\left(V_R(X_{\tau})-{X}_{\tau}{\cal Y}_{\tau}{\cal D}_{\tau-}\right) \right]+yx\right\}\nonumber\\
\le &\sup_{\{\zeta_t\}}\left\{ \mathbb{E}\left[\int_0^{\tau}e^{-\d t}\left\{\left(\tilde{u}\left(\frac{{\cal Y}_t{\cal D}_t}{\k_1}\right)+\e_1{\cal Y}_t{\cal D}_t\right){\bf 1}_{\{\zeta_t=\mathfrak{B}_1 \}}+\left(\tilde{u}\left(\frac{{\cal Y}_t{\cal D}_t}{\k_2}\right)+\e_2{\cal Y}_t{\cal D}_t\right){\bf 1}_{\{\zeta_t=\mathfrak{B}_2 \}}\right\}dt\right.\right.\nonumber\\+&\left.\left.{{\bf 1}_{\{\tau<\infty \}}}e^{-\d \tau}J_R({\cal Y}_{\tau}{\cal D}_{\tau-}) \right]+yx\right\}, \nonumber
\end{align}
where $y>0$ is the Lagrangian multiplier associated with constraint \eqref{eq:static1} {and ${\cal D}_t\in {\cal NI}$}.

Since $\tilde{u}(y)=\sup_{c\ge 0}(u(c)-yc)=u(I(y))-yI(y)$, we deduce that the candidate of optimal consumption $\hat{c}({\cal Y}_t{\cal D}_t)$ for $y>0$ is given by 
	\begin{equation}
	\hat{c}({\cal Z}_t^{{\cal D}}) :=
	\begin{cases}
	\dfrac{1}{\k_1}I\left(\dfrac{{\cal Z}_t^{{\cal D}}}{\k_1}\right)\;\;\;&\mbox{for}\;\;\zeta_t = \mathfrak{B}_1, \vspace{5mm}\\
	\dfrac{1}{\k_2}I\left(\dfrac{{\cal Z}_t^{{\cal D}}}{\k_2}\right)\;\;\;&\mbox{for}\;\;\zeta_t = \mathfrak{B}_2,
	\end{cases}
	\end{equation}
where ${\cal Z}_t^{{\cal D}}:={\cal Y}_t {\cal D}_t$. 

\begin{rem}
	Note that the dynamics of ${\cal Z}_t^{\cal D}$ is given by
	\begin{equation}\label{eq:dynamics-Z}
		\dfrac{d{\cal Z}_t^{\cal D}}{{\cal Z}_t^{\cal D}} =(\d-r)dt -\t dW_t +\dfrac{d{\cal D}_t}{{\cal D}_t}\;\;\mbox{with}\;\;{\cal Z}_0^{\cal D}=y. 
	\end{equation}
\end{rem}

Let us define $f(y)$ by 
$$f(y):= \frac{1}{y}\left(\tilde{u}\left(\frac{y}{\k_1}\right)-\tilde{u}\left(\frac{y}{\k_2}\right)+y (\e_1-\e_2)\right).$$

{Quantity $\tilde{u}\left(\frac{y}{\k_1}\right)-\tilde{u}\left(\frac{y}{\k_2}\right)+y (\e_1-\e_2)$ in the definition of $f(y)$ can be regarded as the difference between the utility value from job $\mathfrak{B}_1$ and that from $\mathfrak{B}_2$ for a given $y$. It compares utility values of consumption and those of income, for the latter of which we use the Lagrange multiplier $y,$  the marginal value of wealth, to transform the monetary value to the utility value.} 

\begin{lem}\label{lem:z_S}
$f(y)$ is strictly increasing in $y>0$ and there exists a unique $z_S>0$ such that $$f(z_S)=0.$$
\end{lem}
\begin{proof}
	Since 
		\begin{align*}
		f'(y) =& -\dfrac{1}{y^2}\left(\tilde{u}\left(\dfrac{y}{\k_1}\right)-\tilde{u}\left(\dfrac{y}{\k_2}\right)\right)+\dfrac{1}{y}\left(-\dfrac{1}{\k_1}I\left(\dfrac{y}{\k_1}\right)+\dfrac{1}{\k_2}I\left(\dfrac{y}{\k_2}\right)\right)\\
		=&-\dfrac{1}{y^2}\left(u\left(\dfrac{y}{\k_1}\right)-u\left(\dfrac{y}{\k_2}\right)\right)>0,
		\end{align*}
$f(y)$ is strictly increasing in $y>0$. 

Moreover, the mean value theorem implies that there exits $\hat{y}\in(y/\k_2, y/\k_1)$ such that 
\begin{equation}
\tilde{u}\left(\frac{y}{\k_1}\right)-\tilde{u}\left(\frac{y}{\k_2}\right)=-I(\hat{y})\left(\dfrac{y}{\k_1}-\dfrac{y}{\k_2}\right).
\end{equation}

Since $\lim_{y\to 0}I(y)=\infty$, $\lim_{y\to\infty}I(y)=0$, and $I(\frac{y}{\k_1})<I(\hat{y})<I(\frac{y}{\k_2})$, 
we deduce that 
\begin{equation}
\lim_{y\to \infty} f(y) = (\e_1-\e_2)>0\;\;\;\mbox{and}\;\;\;\lim_{y\to 0}f(y)=-\infty.
\end{equation}

Thus, the intermediate value theorem implies that there exists a unique $z_S>0$ such that 
$$
f(z_S)=0. 
$$

\end{proof}

For any job process $\zeta_t$, Lemma \ref{lem:z_S} implies that 
\begin{align*}
&\left(\tilde{u}\left(\frac{{\cal Z}_t^{{\cal D}}}{\k_1}\right)+\e_1{\cal Z}_t^{{\cal D}}\right){\bf 1}_{\{\zeta_t=\mathfrak{B}_1 \}}+\left(\tilde{u}\left(\frac{{\cal Z}_t^{{\cal D}}}{\k_2}\right)+\e_2{\cal Z}_t^{{\cal D}}\right){\bf 1}_{\{\zeta_t=\mathfrak{B}_2 \}}\\
\le &\left(\tilde{u}\left(\frac{{\cal Z}_t^{{\cal D}}}{\k_1}\right)+\e_1{\cal Z}_t^{{\cal D}}\right){\bf 1}_{\{{\cal Z}_t^{{\cal D}} \ge z_S \}}+\left(\tilde{u}\left(\frac{{\cal Z}_t^{{\cal D}}}{\k_2}\right)+\e_2{\cal Z}_t^{{\cal D}}\right){\bf 1}_{\{{\cal Z}_t^{{\cal D}} < z_S \}}.
\end{align*}
Hence,  the candidate of optimal job state process $\hat{\zeta}({\cal Z}_t^{{\cal D}})$ is given by 
	\begin{equation}
	\hat{\zeta}({\cal Z}_t^{{\cal D}}) :=
	\begin{cases}
	\mathfrak{B}_1\;\;\;&\mbox{for}\;\;{\cal Z}_t^{{\cal D}}\ge z_S, \vspace{2mm}\\
	\mathfrak{B}_2\;\;\;&\mbox{for}\;\;{\cal Z}_t^{{\cal D}}< z_S.
	\end{cases}
	\end{equation}
{That is, job $\mathfrak{B}_j$ is chosen if  the utility value from it is greater than  that from job $\mathfrak{B}_i$ for $i\ne j$.}

It follows from \eqref{eq:Lagrangian1} that
\begin{align}
\label{eq:Lagrangian2}
\mathfrak{L}\le& \mathbb{E}\left[\int_0^\tau e^{-\d t}\left\{\left(\tilde{u}\left(\frac{{\cal Z}_t^{{\cal D}}}{\k_1} \right)+\e_1 {\cal Z}_t^{{\cal D}}\right){\bf 1}_{\{{\cal Z}_t^{{\cal D}}>z_S\}}+\left(\tilde{u}\left(\frac{{\cal Z}_t^{{\cal D}}}{\k_2} \right)+\e_2 {\cal Z}_t^{{\cal D}}\right){\bf 1}_{\{{\cal Z}_t^{{\cal D}}\le z_S\}}\right\}dt\right.\\+&\left.{{\bf 1}_{\{\tau<\infty \}}}e^{-\d \tau}J_R({\cal Z}_{\tau-}^{{\cal D}})\right] +yx.\nonumber
\end{align}

Minimizing over $y>0$ and ${\cal D}_t$ in \eqref{eq:Lagrangian2} yields 
	\begin{align*}
	\mathfrak{L}&\le \inf_{\{y>0,\;{\cal D}_t\in{\cal NI}\}}\left\{{\cal J}_0(y;{\cal D},\tau) +yx\right\},
	\end{align*}
where
\begin{align*}
{\cal J}_0(y;{\cal D},\tau):=&\mathbb{E}\left[\int_0^\tau e^{-\d t}\left\{\left(\tilde{u}\left(\frac{{\cal Z}_t^{\cal D}}{\k_1} \right)+\e_1 {\cal Z}_t^{\cal D}\right){\bf 1}_{\{{\cal Z}_t^{\cal D}>z_S\}}+\left(\tilde{u}\left(\frac{{\cal Z}_t^{\cal D}}{\k_2} \right)+\e_2 {\cal Z}_t^{\cal D}\right){\bf 1}_{\{{\cal Z}_t^{\cal D}\le z_S\}}\right\}dt\right.\\+&\left.{{\bf 1}_{\{\tau<\infty \}}}e^{-\d \tau}J_R({\cal Z}_{\tau-}^{\cal D})\right].
\end{align*}

Therefore, we deduce that 
\begin{align}
\label{eq:weak-duality}
V(x)&\le \sup_{\tau\in{\cal  S}}\inf_{\{y>0,\;{\cal D}_t\in{\cal NI}\}}\left\{{\cal J}_0(y;{\cal D},\tau) +yx\right\}\\& \le \inf_{y>0}\left(\sup_{\tau\in{\cal S}}\inf_{{\cal D}_t\in{\cal NI}}{\cal J}_0(y;{\cal D},\tau)+yx\right)\nonumber\\&\le \inf_{y>0}\left(\inf_{{\cal D}_t\in{\cal NI}}\sup_{\tau\in{\cal S}}{\cal J}_0(y;{\cal D},\tau)+yx\right).\nonumber
\end{align}

\subsection{A Two-person Zero-sum Game}

 We will now consider  a two-person zero-sum game arising from \eqref{eq:weak-duality}. The game  involves a controller, who is a minimizer and choooses a process ${\cal D}\in{\cal NI}$, and a stopper, who is maximizer and chooses a stopping time $\tau\in{\cal S}$. The two agents share the same performance criterion, which is given by 
\begin{align}
\label{eq:game}
{\cal J}_0(z;{\cal D},\tau)=&\mathbb{E}\left[\int_0^\tau e^{-\d t}\left\{\left(\tilde{u}\left(\frac{{\cal Z}_t^{\cal D}}{\k_1} \right)+\e_1 {\cal Z}_t^{\cal D}\right){\bf 1}_{\{{\cal Z}_t^{\cal D}>z_S\}}+\left(\tilde{u}\left(\frac{{\cal Z}_t^{\cal D}}{\k_2} \right)+\e_2 {\cal Z}_t^{\cal D}\right){\bf 1}_{\{{\cal Z}_t^{\cal D}\le z_S\}}\right\}dt\right.\\+&\left.{{\bf 1}_{\{\tau<\infty \}}}e^{-\d \tau}J_R({\cal Z}_{\tau-}^{\cal D})\right],\nonumber
\end{align}
with ${\cal Z}_0^{\cal D}= z$ (we recall ${\cal Z}_t^{\cal D}={\cal Y}_t{\cal D}_t$).

Next define the lower value $\underline{J}$ and the upper value $\bar{J}$ as 
$$
\underline{J}(z):=\sup_{\tau\in{\cal S}}\inf_{{\cal D}\in{\cal NI}}{\cal J}_0(z;{\cal D},\tau)
$$
and 
$$
\bar{J}(z):=\inf_{{\cal D}\in{\cal NI}}\sup_{\tau\in{\cal S}}{\cal J}_0(z;{\cal D},\tau).
$$

If $\underline{J}(z)=\bar{J}(z)$, then the game is said to have a value and we denote the common value $\underline{J}(z)=\bar{J}(z)$ by $J(z)$. 

 A {\it Nash equilibrium}, or equivalently a {\it saddle point} $(\widehat{\cal D},\hat{\tau})\in{\cal NI}\times {\cal S}$, such that $\widehat{\cal D}$ is the {\it best response} given ${\hat{\tau}}$ while  $\hat{\tau}$ is simultaneously the \textit{best response} given $\widehat{\cal D}$, i.e., 
\begin{equation}
{\cal J}_0 (z; \widehat{\cal D}, \tau) \le {\cal J}_0 (z;\widehat{\cal D}, \hat{\tau}) \le {\cal J}_0 (z;{\cal D}, \hat{\tau}) 
\end{equation}
for any ${\cal D}\in {\cal NI}$, ${\tau}\in{\cal S}$. It is clear that if there exists a saddle point for the game, then the game has a value, i.e., 
\begin{equation}
J(z) = {\cal J}_0 (z,\widehat{\cal D}, \hat{\tau}).
\end{equation}

We can now define the dual problem as  finding a saddle point of the game described above.
\begin{pr}[Game]\label{pr:dual}	
	Find a {\it Nash equilibrium} $(\widehat{\cal D}, \hat{\tau})\in{\cal NI}\times {\cal S}$ such that 
	\begin{equation}\label{eq:NE}
	{\cal J}_0 (z; \widehat{\cal D}, \tau) \le {\cal J}_0 (z;\widehat{\cal D}, \hat{\tau}) \le {\cal J}_0 (z;{\cal D}, \hat{\tau}) \;\;\;\mbox{for any}\;\;{\cal D}\in{\cal NI},\;\;\tau\in{\cal S}.
	\end{equation}
	If a {\it Nash equilibrium}  exists, the dual value function $J(y)$ is defined as 
	\begin{equation}\label{eq:exchange}
	J(z) = \sup_{\tau\in{\cal S}}\inf_{{\cal D}\in{\cal NI}}{\cal J}_0(z;{\cal D},\tau)=\inf_{{\cal D}\in{\cal NI}}\sup_{\tau\in{\cal S}}{\cal J}_0(z;{\cal D},\tau)={\cal J}_0 (z;\widehat{\cal D}, \hat{\tau}) 
	\end{equation}
	with ${\cal Y}_{0}=z$.
\end{pr}

\begin{rem}	
	We can get the following weak duality
		\begin{align}
		\label{eq:weak-duality2}
		V(x)=&\sup_{(c,\pi, \d, \tau)\in \mathcal{A}(x)}\mathbb{E}\left[\int_0^{\tau}e^{-\d t}\left(u(\k_1 c_t){\bf 1}_{\{\zeta_t=\mathfrak{B}_1 \}}+u(\k_2 c_t){\bf 1}_{\{\zeta_t=\mathfrak{B}_2 \}}\right)dt +{{\bf 1}_{\{\tau<\infty \}}} \int_{\tau}^\infty e^{-\d t}u(c_t)dt \right]\\\le &\inf_{y>0} \left(J(y)+yx\right).\nonumber
		\end{align}
	We will show in Theorem \ref{thm:main} that the maximized value is indeed equal to the right-hand side of the last inequality in \eqref{eq:weak-duality2} with infimum being replaced by minimum, i.e.,
	\begin{equation}
	V(x) = \min_{y>0} \left(J(y)+yx\right).
	\end{equation}
\end{rem}

	\begin{rem}
		\citet{FP2007} or \citet{LS2011}\footnote{See the equation (50) in extended appendix of \citet{FP2007}) or the equation (A10) in \citet{LS2011}.} state that the $\sup_{\tau \in{\cal S}}$ and $\inf_{{\cal D}\in{\cal NI}}$ in \eqref{eq:exchange} can be interchanged. However, they do not provide a theoretical proof of the fact. We provide a rigorous proof  by considering the two person zero-sum game in Problem \ref{pr:dual}.
	\end{rem}

\section{Optimal Strategies}\label{sec:optimal strategies}
In this section we derive optimal strategies for the main optimization problem (Problem \ref{pr:main}) by obtaining the value of the two-person zero-sum game in Problem \ref{pr:dual}.

\subsection{Heuristic derivation of Hamilton-Jacobi-Bellman quasi-variational inequality(HJBQV) for the dual value function $J$}

{In this subsection, we derive HJBQV for $J(z),$ the value function of the game, by relying on heuristic and intuitive arguments. We will derive a concrete solution to the HJBQV and provide the verification that the solution is equal to the value function of the game in later subsections. }

If a {\it Nash-equilibrium} $(\widehat{\cal D}, \hat{\tau})\in{\cal NI}\times {\cal S}$ of the game \eqref{eq:game} exists, then the lower value $\underline{J}$ and the upper value $\bar{J}$ can be written as 
\begin{align*}
\underline{J}(z)= \sup_{\tau\in{\cal S}}&\mathbb{E}\left[\int_0^\tau e^{-\d t}\left\{\left(\tilde{u}\left(\frac{{\cal Z}_t^{\widehat{\cal D}}}{\k_1} \right)+\e_1 {\cal Z}_t^{\widehat{\cal D}}\right){\bf 1}_{\{{\cal Z}_t^{\widehat{\cal D}}>z_S\}}+\left(\tilde{u}\left(\frac{{\cal Z}_t^{\widehat{\cal D}}}{\k_2} \right)+\e_2 {\cal Z}_t^{\widehat{\cal D}}\right){\bf 1}_{\{{\cal Z}_t^{\widehat{\cal D}}\le z_S\}}\right\}dt\right.\\+&\left.{{\bf 1}_{\{\tau<\infty \}}}e^{-\d \tau}J_R({\cal Z}_{\tau-}^{\widehat{\cal D}})\right]
\end{align*}
and
\begin{align*}
\bar{J}(z)= \inf_{{\cal D}\in{\cal NI}}&\mathbb{E}\left[\int_0^{\hat{\tau}} e^{-\d t}\left\{\left(\tilde{u}\left(\frac{{\cal Z}_t^{\cal D}}{\k_1} \right)+\e_1 {\cal Z}_t^{\cal D}\right){\bf 1}_{\{{\cal Z}_t^{\cal D}>z_S\}}+\left(\tilde{u}\left(\frac{{\cal Z}_t^{\cal D}}{\k_2} \right)+\e_2 {\cal Z}_t^{\cal D}\right){\bf 1}_{\{{\cal Z}_t^{\cal D}\le z_S\}}\right\}dt \right.\\+&\left.{{\bf 1}_{\{\hat{\tau}<\infty \}}}e^{-\d \tau}J_R({\cal Z}_{\hat{\tau}-}^{\cal D})\right].
\end{align*}

Considering dynamic programming principle, we expect that $\underline{J}$ satisfies the following HJB equation in the region $\{z>0 \mid \underline{J}'(z)<0 \}$:
	\begin{equation}\label{eq:VI1}
	\max\left\{{\cal L}{\cal Q}(z) +\left(\tilde{u}\left(\frac{z}{\k_1} \right)+\e_1 z\right){\bf 1}_{\{z>z_S\}}+\left(\tilde{u}\left(\frac{z}{\k_2} \right)+\e_2 z\right){\bf 1}_{\{z\le z_S\}} , J_R(z)-{\cal Q}(z)\right\}=0,
	\end{equation}
where the differential operator ${\cal L}$ is given by 
\begin{equation}
{\cal L}=\dfrac{\t^2}{2}z^2\dfrac{d^2}{dz^2} + (\d-r)z\dfrac{d}{dz}-\d. 
\end{equation}

In the region $\{z>0 \mid \bar{J}(z)>J_R(z) \}$, we also expect that $\bar{J}$ satisfies the following HJB equation with a gradient constraint 
	\begin{equation}\label{eq:VI2}
	\min\left\{{\cal L}{\cal Q}(z) +\left(\tilde{u}\left(\frac{z}{\k_1} \right)+\e_1 z\right){\bf 1}_{\{z>z_S\}}+\left(\tilde{u}\left(\frac{z}{\k_2} \right)+\e_2 z\right){\bf 1}_{\{z\le z_S\}} , -{\cal Q}'(z)\right\}=0.
	\end{equation}
Since $J(z)=\bar{J}(z)=\underline{J}(z)$ when a {\it Nash-equilibrium} exists, in view of \eqref{eq:VI1} and \eqref{eq:VI2}, we expect that the dual value function $J(z)$ is  a solution ${\cal Q}(z)$ to the following {max-min} type of Hamilton-Jacobi-Bellman quasi-variational inequality (HJBQV) arising from Problem \ref{pr:dual}: for $z>0$ 
	\begin{equation}\label{eq:HJBQV}
	\max\left\{\min\left\{{\cal L}{\cal Q}(z)+\left(\tilde{u}\left(\frac{z}{\k_1} \right)+\e_1 z\right){\bf 1}_{\{z>z_S\}}+\left(\tilde{u}\left(\frac{z}{\k_2} \right)+\e_2 z\right){\bf 1}_{\{z\le z_S\}} , -{\cal Q}'(z) \right\}, J_R(z)-{\cal Q}(z) \right\}=0.
	\end{equation}

\subsection{Solution to HJBQV \eqref{eq:HJBQV}}

To find a solution to  HJBQV \eqref{eq:HJBQV}, we will employ a {\it guess-and-verify} approach. We will provide formal results in the next section.

From the perspective of the stopper in \eqref{eq:HJBQV}, the state space $\mathbb{R}_+$ splits into  two regions:
\begin{align*}
{\bf RR}&:=\{z>0\mid {\cal Q}(z)=J_R(z)   \}\;\;\;(\mbox{retirement region}),\\
{\bf WR}&:=\{z>0\mid {\cal Q}(z)>J_R(z)   \}\;\;\;(\mbox{working region}).
\end{align*}

Similarly, from the perspective of the controller, the state space $\mathbb{R}_+$  can be decomposed into two regions:
\begin{align*}
{\bf IR}&:=\{z>0\mid \;{\cal Q}'(z)<0 \}\;\;\;(\mbox{in-action region}),\\
{\bf AR}&:=\{z>0\mid \;{\cal Q}'(z)=0 \}\;\;\;(\mbox{adjustment region}).
\end{align*}

In the retirement region {\bf RR}, ${\cal Q}(z)=J_R(z)$ satisfies 
$$
\min\left\{{\cal L}J_R(z)+\left(\tilde{u}\left(\frac{z}{\k_1} \right)+\e_1 z\right){\bf 1}_{\{z>z_S\}}+\left(\tilde{u}\left(\frac{z}{\k_2} \right)+\e_2 z\right){\bf 1}_{\{z\le z_S\}} , -\dfrac{dJ_R}{dz}(z) \right\}\le 0. 
$$ 
It follows from \eqref{eq:J0-prime} that 
$$
J_R'(y)=\dfrac{2}{\t^2(n_1-n_2)}\left[y^{n_2-1}\int_0^y \eta^{-n_2}I(\eta)d\eta +y^{n_1-1}\int_y^\infty \eta^{-n_1}I(\eta)d\eta\right]<0.
$$

Thus, we have 
$$
{\cal L}J_R(z)+\left(\tilde{u}\left(\frac{z}{\k_1} \right)+\e_1 z\right){\bf 1}_{\{z>z_S\}}+\left(\tilde{u}\left(\frac{z}{\k_2} \right)+\e_2 z\right){\bf 1}_{\{z\le z_S\}}\le 0\;\;\mbox{for}\;\;z\in{\bf RR}.
$$

Since ${\cal L}J_R(z) + \tilde{u}(z)=0$, we deduce that 
$$
h(z) \le 0 \;\;\mbox{for}\;\;z\in {\bf RR},
$$
where 
$$
h(z):= \left(\tilde{u}\left(\frac{z}{\k_1} \right)-\tilde{u}(z)+\e_1 z\right){\bf 1}_{\{z>z_S\}}+\left(\tilde{u}\left(\frac{z}{\k_2} \right)-\tilde{u}(z)+\e_2 z\right){\bf 1}_{\{z\le z_S\}}.
$$

{Function $h(z)$ gives the difference in the utility value after retirement and that before retirement for a given $z$. Note that $z$ is used to transform the monetary value of income $\e_i \ (i=1,2)$ to the utility value.}

\begin{lem}~\label{lem:h} 
	\begin{itemize}
		\item[(a)] ${h(z)}/{z}$ is strictly increasing in $z>0$. 
		\item[(b)] There exists a unique $\hat{z}>0$ such that $ h(\hat{z})=0.$
		\item[(c)] $h(z)<0$ for $z\in(0,\hat{z})$, and $h(z)>0$ for $z\in(\hat{z},\infty)$.
	\end{itemize}
\end{lem}
\begin{proof}
	The proof is  almost identical with that of Lemma \ref{lem:z_S} and we omit its detail.
\end{proof}

\begin{rem}
	Lemma \ref{lem:h} implies that $z\le \hat{z}$ for $z\in{\bf RR}$. 
	Moreover, the inequality ${z}_S>\hat{z}$ holds if and only if  {$h(z_S)>0$}. 
\end{rem}

Suppose that there exist two boundaries $z_R\in(0,\hat{z})$ and $z_B\in(z_R,\infty)$ such that the agent chooses the option to retire  if ${\cal Z}_t \le z_R$ and the agent's wealth is zero if ${\cal Z}_t \ge z_B$. That is, 
\begin{eqnarray}
\begin{split}
\begin{cases}
{\cal Q}(z) = J_R(z)\;\;\;&\mbox{if}\;\;z\le z_R,\\
{\cal Q}'(z) = 0\;\;\;&\mbox{if}\;\;z\ge z_B,
\end{cases}
\end{split}
\end{eqnarray}
and it follows from the smooth-pasting(or super-contact) condition (see \citet{Dumas1989}) that
\begin{equation}\label{eq:smooth-pasting}
{\cal Q}(z_R)=J_R(z_R),\;\;{\cal Q}'(z_R)=J_R'(z_R),\;\; {\cal Q}'(z_B)=0,\;\;\mbox{and}\;\;{\cal Q}''(z_B)=0.
\end{equation}

\begin{rem}\label{rem:saddle}
	We will show later that a {\it Nash-equilibrium} for the game in \eqref{eq:game} is given by a pair of {\it barrier strategies} $({\cal D}^{z_B}, \tau_{z_R})$ for $z_R<z_B$, where we define 
	\begin{equation*}
	{\cal D}_t^{z_B} := \min\left\{1, \inf_{0\le s \le t }\dfrac{z_B}{{\cal Y}_s}\right\}\;\;\mbox{for}\;\;t\ge 0,\;\;\;\tau_{z_R}:=\inf\{t\ge 0 \mid {\cal Z}_t^{{\cal D}^{z_B}} < z_R \}
	\end{equation*}
	That is, 
	\begin{equation}
	{\cal Q}(z) = J(z) = {\cal J}_0(z;{\cal D}^{z_B}, \tau_{z_R}).
	\end{equation}
	
\end{rem}

In the region ${\bf WR}\cap{\bf IR}$, ${\cal Q}(z)$ satisfies 
\begin{equation}\label{eq:ODE_Q}
{\cal L} {\cal Q}(z)+ h(z) + \tilde{u}(z) =0.
\end{equation}

A general solution to the equation \eqref{eq:ODE_Q} can be written as the sum of a general solution
to the homogeneous equation and a particular solution:
	\begin{align*}
		{\cal Q}(z) =& E_1 z^{n_1} + E_2 z^{n_2} +\dfrac{2}{\t^2(n_1-n_2)} {\left[z^{n_2}\int_0^z \eta^{-n_2-1}(h(\eta)+\tilde{u}(\eta))d\eta +z^{n_1}\int_z^\infty \eta^{-n_1-1}(h(\eta)+\tilde{u}(\eta))d\eta\right]}\\
		=& E_1 z^{n_1} + E_2 z^{n_2}+ J_R(z) + \Psi_h(z),
	\end{align*}
where 
$$
\Psi_h(z):=\dfrac{2}{\t^2(n_1-n_2)}\left[z^{n_2}\int_0^z \eta^{-n_2-1}h(\eta)d\eta +z^{n_1}\int_z^\infty \eta^{-n_1-1}h(\eta)d\eta\right].
$$

Since ${\cal Q}(z_R)=J_R(z_R)$ and ${\cal Q}'(z_R)=J_R'(z_R)$, we have 
\begin{align*}
	&E_1 z_R^{n_1} + E_2 z_R^{n_2}+ J_R(z_R) + \Psi_h(z_R)=J_R(z_R),\\
	&n_1 E_1 z_R^{n_1-1} + n_2 E_2 z_R^{n_2-1}+ J_R'(z_R) + \Psi_h'(z_R)=J_R'(z_R),
\end{align*}
which implies that 
\begin{equation}\label{eq:coefficient-1}
E_1 =-\dfrac{2}{\t^2(n_1-n_2)}\int_{z_R}^\infty \eta^{-n_1-1}h(\eta)d\eta\;\;\;\mbox{and}\;\;\;E_2 =-\dfrac{2}{\t^2(n_1-n_2)}\int_0^{z_R} \eta^{-n_2-1}h(\eta)d\eta.
\end{equation}

From \eqref{eq:utilde-integral}, we can easily deduce that 
$$
\int_0^z \eta^{-n_2-1}|h(\eta)|d\eta + \int_z^\infty \eta^{-n_1-1}|h(\eta)|d\eta {<\infty}.
$$
Proposition \ref{pro:review:KMZ} implies that 
$$
\liminf_{y\downarrow 0}z^{-n_2}|h(z)|=\liminf_{z\uparrow \infty}z^{-n_1}|h(z)|=0.
$$

Since 
$$
J_R'(y)=\dfrac{2}{\t^2(n_1-n_2)}\left[y^{n_2-1}\int_0^y \eta^{-n_2}I(\eta)d\eta +y^{n_1-1}\int_y^\infty \eta^{-n_1}I(\eta)d\eta\right],
$$
by  integration by parts 
\begin{align*}
J_R'(z)+\Psi_h'(z)=&\dfrac{2}{\t^2(n_1-n_2)}\left[z^{n_2-1}\int_0^z \eta^{-n_2}(h'(\eta)-I(\eta))d\eta +z^{n_1-1}\int_z^\infty \eta^{-n_1}(h'(\eta)-I(\eta))d\eta\right]\\
=&\dfrac{2}{\t^2(n_1-n_2)}\left[z^{n_2-1}\int_0^z \eta^{-n_2}(\hat{\e}(\eta)-\hat{c}(\eta))d\eta +z^{n_1-1}\int_z^\infty \eta^{-n_1}(\hat{\e}(\eta)-\hat{c}(\eta))d\eta\right],
\end{align*}
where $\hat{\e}(z):=\e_1{\bf 1}_{\{z > z_S\}} +\e_2{\bf 1}_{\{z\le z_S\}}$.

Since ${\cal Q}'(z_B)={\cal Q}''(z_B)=0$, we have 
\begin{align*}
&n_1E_1 z_B^{n_1-1} + n_2E_2 z_B^{n_2-1}+ J_R'(z_R) + \Psi_h'(z_R)=0,\\
&n_1(n_1-1) E_1 z_B^{n_1-2} + n_2(n_2-1) E_2 z_R^{n_2-2}+ J_R''(z_R) + \Psi_h''(z_R)=0,
\end{align*}
which implies that 
\begin{equation}\label{eq:coefficient-2}
E_1 =-\dfrac{2}{n_1\t^2(n_1-n_2)}\int_{z_R}^\infty \eta^{-n_1}(\hat{\e}(\eta)-\hat{c}(\eta))d\eta\;\;\;\mbox{and}\;\;\;E_2 =-\dfrac{2}{n_2\t^2(n_1-n_2)}\int_0^{z_R} \eta^{-n_2}(\hat{\e}(\eta)-\hat{c}(\eta))d\eta.
\end{equation}

From $\eqref{eq:coefficient-1}$ and $\eqref{eq:coefficient-2}$, we deduce that $z_B$ and $z_R$ satisfy the  coupled algebraic equations:
$$
\phi_1(z_R,z_B)=0,\;\;\phi_2(z_R,z_B)=0,
$$
where 
\begin{align*}
\phi_1(z_1,z_2):&=\int_{z_2}^\infty \eta^{-n_1}(\hat{\e}(\eta)-\hat{c}(\eta))d\eta-n_1\int_{z_1}^\infty \eta^{-n_1-1}h(\eta)d\eta,\\
\phi_2(z_1,z_2):&=\int_0^{z_2}\eta^{-n_2}(\hat{\e}(\eta)-\hat{c}(\eta))d\eta-n_2\int_0^{z_1}\eta^{-n_2-1}h(\eta)d\eta.
\end{align*}
In summary, we set ${\cal Q}(z),$ a candidate solution to the HJBQV \eqref{eq:HJBQV}, as follows: 
\begin{eqnarray}
\begin{split}\label{eq:Q-sol-HJBQV}
{\cal Q}(z)=
\begin{cases}
J_R(z)\;\;\;&\mbox{for}\;\;z\le z_R,\\
E_1 z^{n_1} + E_2 z^{n_2} + \phi_h(z)+J_R(z)\;\;\;&\mbox{for}\;\;z_R \le z \le z_B,\\
E_1 z_B^{n_1} + E_2 z_B^{n_2} + \phi_h(z_B)+J_R(z_B)\;\;\;&\mbox{for}\;\;\;z_B \le  z. 
\end{cases}
\end{split}
\end{eqnarray}

From \eqref{eq:coefficient-1} and \eqref{eq:coefficient-2},  we deduce that for $z_R<z<z_B$
	\begin{eqnarray}
	\begin{split}\label{eq:Q-1}
	{\cal Q}(z) =\dfrac{2}{\t^2(n_1-n_2)}\left[z^{n_2}\int_{z_R}^z \eta^{-n_2-1}h(\eta)d\eta +z^{n_1}\int_z^{z_R} \eta^{-n_1-1}h(\eta)d\eta\right]+J_R(z)
	\end{split}
	\end{eqnarray}
and 
	\begin{eqnarray}
	\begin{split}\label{eq:Q-2}
	{\cal Q}'(z) =\dfrac{2}{\t^2(n_1-n_2)}\left[z^{n_2-1}\int_{z_B}^z \eta^{-n_2}(\hat{\e}(\eta)-\hat{c}(\eta))d\eta +z^{n_1-1}\int_z^{z_B}\eta^{-n_1}(\hat{\e}(\eta)-\hat{c}(\eta))d\eta\right]
	\end{split}
	\end{eqnarray}
{
\begin{rem}
	According to \eqref{eq:Q-sol-HJBQV}, the candidate solution has two components: (i) $J_R(z),$ which is the present value of the convex conjugate of consumption and equal to the dual value function after retirement, and (ii) the difference $Q(z)-J_R(z)$ which can be regarded as the option value of retirement. In the two-person zero-sum game, the stopper attempts to maximize the option value, while the controller tries to minimize it. 
\end{rem}
}
\subsection{Verification I: Dual Problem}

\begin{pro}\label{pro:two-free-boundaries}
	The coupled algebraic equations $\phi_1(z_R,z_B)=0$ and $\phi_2(z_R,z_B)=0$ have a unique pair $(z_R,z_B)$ such that $0<z_R<\hat{z}$ and $z_B>z_R$.
\end{pro}
\begin{proof}

Since $\hat{\e}(\eta)-\hat{c}(\eta)$ is strictly increasing in $\eta>0$ and  $\lim_{\eta\to 0}\hat{\e}(\eta)-\hat{c}(\eta) = -\infty$, $\lim_{\eta \to \infty} \hat{\e}(\eta)-\hat{c}(\eta)=\e_1$, there exists a unique $\bar{z}>0$ such that 
$$
\hat{\e}(\bar{z})-\hat{c}(\bar{z})=0. 
$$

From
$$
\int_0^{z_B}\eta^{-n_2}(\hat{\e}(\eta)-\hat{c}(\eta))d\eta=n_2\int_0^{z_R}\eta^{-n_2-1}h(\eta)d\eta, 
$$
we deduce that 
$$
z_B > \bar{z} \;\;\;\mbox{for}\;\;0<z_R\le \hat{z}. 
$$

Note that for $0<z_R\le \hat{z}$, $z_R<z_B$, 
\begin{align*}
\phi_2(z_R,z_R)&=\int_0^{z_R}\eta^{-n_2}(\hat{\e}(\eta)-\hat{c}(\eta))d\eta-n_2\int_0^{z_R}\eta^{-n_2-1}h(\eta)d\eta\\
&=\int_0^{z_R}\eta^{-n_2}(\hat{\e}(\eta)-\hat{c}(\eta))d\eta+z_R^{-n_2}h(z_R)-\int_0^{z_R}\eta^{-n_2}(\hat{\e}(\eta)-\hat{c}(\eta)+I(\eta))d\eta \\
&=z_R^{-n_2}h(z_R)-\int_0^{z_R}\eta^{-n_2}I(\eta)d\eta <0,
\end{align*}
where we have used  integration by parts in the second equality. 

Since $\lim_{\eta \to \infty}(\hat{\e}(\eta)-\hat{c}(\eta))=\e_1$, there exists a sufficiently large $M>0$ such that 
$$
\hat{\e}(\eta)-\hat{c}(\eta)\ge \frac{\e_1}{2}\;\;\mbox{for}\;\;\eta\ge M. 
$$

Note that 
$$
\int_{M}^\infty \eta^{-n_2}(\hat{\e}(\eta)-\hat{c}(\eta))d\eta \ge \dfrac{\e_1}{2} \int_{M}^\infty \eta^{-n_2}d\eta =\infty.
$$

It follows that 
\begin{equation}
\lim_{z_B \to \infty}\phi_2(z_R,z_B) = \infty. 
\end{equation}
{
Since 
$$
\dfrac{\partial \phi_2}{\partial z_2}(z_R,z_B)=z_B^{-n_2}(\hat{\e}(z_B)-\hat{c}(z_B))>0\;\;\mbox{for}\;\;0<z_R\le \hat{z},
$$ }
the intermediate value theorem implies that for given $0<z_R\le \hat{z}$ there exists a unique $\vartheta(z_R)>z_R$ such that 
$$
\phi_2(z_R, \vartheta(z_R))=0. 
$$
Note that $0<\hat{z} < \vartheta(\hat{z}).$

Thus, we have 
\begin{align}
\label{eq:1}
\phi_1(\hat{z},\vartheta(\hat{z}))=&\int_{\vartheta(\hat{z})}^\infty (\hat{\e}(\eta)-\hat{c}(\eta))d\eta -n_1\int_{\hat{z}}^\infty \eta^{-n_1-1}h(\eta)d\eta \\
=&\int_{\vartheta(\hat{z})}^\infty \eta^{-n_1} (\hat{\e}(\eta)-\hat{c}(\eta)+I(\eta))d\eta -\int_{\vartheta(\hat{z})}^\infty \eta^{-n_1}I(\eta)d\eta-n_1\int_{\hat{z}}^\infty \eta^{-n_1-1}h(\eta)d\eta \nonumber\\
=&-\vartheta(\hat{z})^{-n_1}h(\vartheta(\hat{z}))-n_1\int_{\vartheta(\hat{z})}^\infty \eta^{-n_1-1}h(\eta)d\eta-\int_{\vartheta(\hat{z})}^\infty \eta^{-n_1}I(\eta)d\eta-n_1\int_{\hat{z}}^\infty \eta^{-n_1-1}h(\eta)d\eta\nonumber\\
=&-\vartheta(\hat{z})^{-n_1}h(\vartheta(\hat{z}))-n_1\int_{\hat{z}}^{\vartheta(\hat{z})} \eta^{-n_1-1}h(\eta)d\eta-\int_{\vartheta(\hat{z})}^\infty \eta^{-n_1}I(\eta)d\eta<0,\nonumber
\end{align}
where we have used  integration by parts in the third equality and $0=h(\hat{z})<h(\vartheta(\hat{z}))$. 

Since
$$
\int_0^{\bar{z}}\eta^{-n_2}(\hat{\e}(\eta)-\hat{c}(\eta))d\eta <0,\;\;\;\lim_{z\to \infty}\int_0^{z}\eta^{-n_2}(\hat{\e}(\eta)-\hat{c}(\eta))d\eta=\infty,
$$
and 
$$
\dfrac{d}{dz}\left(\int_0^{z}\eta^{-n_2}(\hat{\e}(\eta)-\hat{c}(\eta))d\eta\right)=z^{-n_2}(\hat{\e}(z)-\hat{c}(z))>0\;\;\mbox{for}\;\;z>\bar{z}, 
$$
we deduce that there exists a unique $\underline{z}>\bar{z}$ such that 
$$
\int_0^{\underline{z}}\eta^{-n_2}(\hat{\e}(\eta)-\hat{c}(\eta))d\eta=0. 
$$

Letting $z\to 0+$, 
$$
0=\lim_{z\to0+}n_2\int_0^{z}\eta^{-n_2-1}h(\eta)d\eta=\lim_{z\to0+}\int_0^{\vartheta(z)}\eta^{-n_2}(\hat{\e}(\eta)-\hat{c}(\eta))d\eta.
$$
That is, $\vartheta(0+)=\lim_{z\to0+} \vartheta(z)=\underline{z}$. 

Hence, we have 
\begin{align*}
\lim_{z\to 0+}\phi_1(z,\vartheta(z))=&\int_{\underline{z}}^\infty \eta^{-n_1}(\hat{\e}(\eta)-\hat{c}(\eta))d\eta - n_1 \lim_{z\to 0+}\int_z^\infty \eta^{-n_1-1}h(\eta)d\eta\\
&>- n_1 \lim_{z\to 0+}\int_z^\infty \eta^{-n_1-1}h(\eta)d\eta.
\end{align*}

For a sufficiently small $0<\nu<\hat{z}$, it follows from Lemma \ref{lem:h} that 
\begin{align*}
\lim_{z\to 0+}\int_z^{\nu} \eta^{-n_1-1}h(\eta)d\eta=\lim_{z\to 0+}\int_z^{\nu} \eta^{-n_1}\dfrac{h(\eta)}{\eta}d\eta  <\dfrac{h(\nu)}{\nu}\lim_{z\to 0+}\int_z^{\nu} \eta^{-n_1}d\eta=-\infty.
\end{align*}
Thus, 
\begin{equation}\label{eq:2}
\lim_{z\to 0+}\phi_1(z,\vartheta(z))=\infty. 
\end{equation}

Note that 
$$
0=\dfrac{\partial }{\partial z_R}(\phi_2(z_R,\vartheta(z_R)))=\dfrac{\partial \phi_2}{\partial z_1}+\dfrac{\partial \phi_2}{\partial z_2}\dfrac{d \vartheta (z_R)}{d z_R}. 
$$

Since $$\dfrac{\partial \phi_2}{\partial z_1}(z_R,z_B)= -n_2 z_R^{-n_2-1}h(z_R)\;\;\mbox{and}\;\;\dfrac{\partial \phi_2}{\partial z_2}(z_R,z_B)= z_B^{-n_2}(\hat{\e}(z_B)-\hat{c}(z_B)),$$
we have 
$$
\dfrac{d \vartheta (z_R)}{d z_R}=\dfrac{n_2 z_R^{-n_2-1}h(z_R)}{z_B^{-n_2}(\hat{\e}(z_B)-\hat{c}(z_B))}.
$$

For $0<z_R<\hat{z}$, we deduce that 
\begin{align}
\label{eq:3}
\dfrac{d }{d z_R}\phi_1(z_R,\vartheta(z_R))=&\dfrac{\partial \phi}{\partial z_1}(z_R,\vartheta(z_R))+\dfrac{\partial \phi}{\partial z_2}(z_R,\vartheta(z_R))\dfrac{d\vartheta(z_R)}{dz_R}\\
=&n_1 z_R^{-n_1-1}h(z_R) -z_B^{-n_1}(\hat{\e}(z_B)-\hat{c}(z_B))\dfrac{n_2 z_R^{-n_2-1}h(z_R)}{z_B^{-n_2}(\hat{\e}(z_B)-\hat{c}(z_B))} \nonumber\\
=&z_R^{-n_1-1}h(z_R)(n_1-n_2 z_B^{n_2-1}z_R^{n_1-n_2})<0 \nonumber
\end{align}

From \eqref{eq:1}, \eqref{eq:2}, and \eqref{eq:3}, we conclude that there exists a unique $0<z_R<\hat{z}$ such that 
$$
\phi_1(z_R,\vartheta(z_R))=0.
$$

Thus, there exists a unique pair $(z_R, z_B$ such that $0<z_R<\hat{z}$, $z_B>z_R$, and 
$$
\phi_1(z_R,z_B) = \phi_2(z_R,z_B)=0.
$$

\end{proof}

{The fact $z_R<\hat{z}$ in Proposition \ref{pro:two-free-boundaries}  implies that optimal retirement decision is made only when $h(z_t)<0$, i.e., the instantaneous utility value after retirement is strictly smaller than that before retirement. This is consistent with the observation that it is optimal to exercise an option only when it is in the money (\citet{DixitPindyck}).}

\begin{pro}~\label{pro:properties-Q} 
\begin{itemize}
	\item[(a)] 	${\cal Q}(z)$ given in \eqref{eq:Q-sol-HJBQV} satisfies the HJBQV \eqref{eq:HJBQV}. Moreover, ${\cal Q}(z)$ is a continuously differentiable function for $z>0$ and twice continuously differentiable function for $z\in (0,\infty)\backslash \{z_R\}$. 
	\item[(b)]  The four regions ${\bf WR}$, ${\bf RR}$, ${\bf AR}$, and ${\bf IR}$ can be rewritten as follows: {
	\begin{align*}
	{\bf RR}&=\{z>0\mid z\le z_R \},\;\;{\bf WR}=\{z>0\mid z>z_R   \},\\
	{\bf AR}&=\{z>0\mid z_B \le z \},\;\;{\bf IR}=\{z>0 \mid 0<z<z_B \}.
	\end{align*} }
    \item[(c)] ${\cal Q}(z)$ is strictly convex in $z\in(0,z_B)$. 
    \item[(d)] 
    $$
    \lim_{z\to z_B-} {\cal Q}'(z) = 0\;\;\;\mbox{and}\;\;\;\lim_{z\to +0}{\cal Q}'(z) = -\infty.
    $$
\end{itemize}
\end{pro}
\begin{proof}
	
\noindent{\bf  (a)} By  construction of ${\cal Q}(z)$ in \eqref{eq:Q-sol-HJBQV}, we can easily confirm that ${\cal Q}(z)$ is a continuously differentiable function for $z>0$ and twice continuously differentiable function for $z\in (0,\infty)\backslash \{z_R\}$. 	

We will prove that ${\cal Q}(z)$ in \eqref{eq:Q-sol-HJBQV} satisfies the HJBQV \eqref{eq:HJBQV}. 

\begin{itemize}
	\item[(i)] the case $z\in (0,z_R]$. 
	
	Since ${\cal Q}(z)=J_R(z)$ for $z\in(0,z_R]$, we have  
	$$
	-\dfrac{d{\cal Q}}{d z} = -J_R'(z)> 0. 
	$$
	
	It follows from ${\cal L}J_R +\tilde{u}(z)=0$ that 
	$$
	{\cal L}{\cal Q}(z) + \tilde{u}(z)+h(z) = {\cal L}J_R +\tilde{u}(z)+h(z) = h(z) < 0 \;\;\mbox{for}\;\;z\in(0,z_R]. 
	$$
	
	Thus, we deduce that 
	$$
	\min\left\{{\cal L}{\cal Q}(z)+\tilde{u}(z)+h(z) , -\dfrac{d{\cal Q}}{dz}(z) \right\}< 0. 
	$$ 
	That is, 
		$$
	\max\left\{\min\left\{{\cal L}{\cal Q}(z)+\tilde{u}(z)+h(z) , -\dfrac{d{\cal Q}}{dz}(z) \right\}, J_R(z)-{\cal Q}(z)\right\}= 0. 
	$$

   \item[(ii)] the case  $z\in (z_R,z_B).$
   
 By  construction of ${\cal Q}(z)$, 
 $$
 {\cal L}{\cal Q}(z) + \tilde{u}(z)+h(z)=0. 
 $$
 
 From \eqref{eq:Q-1}, we deduce that for $z_R<z<z_B$, 
 	\begin{align*}
 	{\cal Q}(z)-J_R(z) =&\dfrac{2}{\t^2(n_1-n_2)}\left[z^{n_2}\int_{z_R}^z \eta^{-n_2-1}h(\eta)d\eta +z^{n_1}\int_z^{z_R} \eta^{-n_1-1}h(\eta)d\eta\right]\\
 	=&\dfrac{2}{\t^2(n_1-n_2)}\left[z^{n_2}\int_{z_R}^z \eta^{-n_2-1}h(\eta)d\eta -z^{n_1}\int_{z_R}^z \eta^{-n_1-1}h(\eta)d\eta\right].
 	\end{align*}
This leads to
	\begin{align*}
	({\cal Q}-J_R)'(z)=&\dfrac{2}{\t^2(n_1-n_2)}\left[n_2 z^{n_2-1}\int_{z_R}^z \eta^{-n_2-1}h(\eta)d\eta -n_1 z^{n_1-1}\int_{z_R}^z \eta^{-n_1-1}h(\eta)d\eta\right].
	\end{align*}

Since $h(z)<0$ for $0<z<\hat{z}$ and $h(z)>0$ for $z>\hat{z}$, $({\cal Q}-J_R)'(z)$ is increasing in $z\in(z_R,\hat{z})$ and decreasing in $z\in(\hat{z},z_B)$. 

It follows from ${\cal Q}(z_R)=J_R(z_R)$ and $Q'(z_B) =0$ that 
$$
({\cal Q}-J_R)(z_R)=0
$$
and 
$$
({\cal Q}-J_R)'(z_B)=-J_R'(z_B) > 0.
$$
	 
Hence, we  conclude that 
$$
({\cal Q}-J_R)'(z)>0\;\;\;\mbox{for}\;\;z\in(z_R,z_B].
$$	 
It follows from ${\cal Q}(z_R)=J_R(z_R)$ that 
\begin{equation}\label{eq:inq:Q:J0}
{\cal Q}(z) > J_R(z)\;\;\mbox{for}\;\;z\in(z_R,z_B].
\end{equation}

From \eqref{eq:Q-2}, 
	\begin{align*}
	J'(z)=\dfrac{2}{\t^2(n_1-n_2)}\left[z^{n_2-1}\int_{z_B}^z \eta^{-n_2}(\hat{\e}(\eta)-\hat{c}(\eta))d\eta +z^{n_1-1}\int_z^{z_B}\eta^{-n_1}(\hat{\e}(\eta)-\hat{c}(\eta))d\eta\right]
	\end{align*}
for $ z_R<z<z_B$.

Hence, we have 	
	\begin{align*}
	J''(z)=\dfrac{2}{\t^2(n_1-n_2)}\left[(1-n_2)z^{n_2-2}\int_{z}^{z_B}\eta^{-n_2}(\hat{\e}(\eta)-\hat{c}(\eta))d\eta +(1-n_1)z^{n_1-2}\int_z^{z_B}\eta^{-n_1}(\hat{\e}(\eta)-\hat{c}(\eta))d\eta\right]
	\end{align*}
for $z_R<z<z_B$. 

Since $\hat{\e}(\eta)-\hat{c}(\eta)>0$ for $\eta>\bar{z}$ and $\hat{\e}(\eta)-\hat{c}(\eta)<0$ for $\eta<\bar{z}$, $J''(z)$ is increasing in $z\in(z_R,\bar{z})$ and decreasing in $z\in(\bar{z},z_B)$. 

Note that for $z_R<z<z_B$, 
$$
\dfrac{\t^2}{2}z^2 {\cal Q}''(z)+(\d-r)z{\cal Q}'(z)-\d {\cal Q}(z) +\tilde{u}(z)+h(z)=0.
$$

Letting $z \to z_R+$, we derive that 
	\begin{align*}
	0=&\lim_{z\to z_R+}\left(\dfrac{\t^2}{2}z^2 {\cal Q}''(z)+(\d-r)z{\cal Q}'(z)-\d {\cal Q}(z) +\tilde{u}(z)+h(z)\right)\\
	=&\dfrac{\t^2}{2}z_R^2\lim_{z\to z_R+}{\cal Q}''(z)-\dfrac{\t^2}{2}z_R^2{\cal J}_0''(z)+h(z_R),
	\end{align*}
where we have used the fact that ${\cal Q}(z)$ is continuously differentiable and ${\cal L}J_R(z_R)+\tilde{u}(z_R)=0$. 

It follows from the strictly convexity of $J_R$ that 
\begin{align*}
{\cal Q}''(z_R+) =J_R''(z_R)-\dfrac{2}{z_R^2\t^2}h(z_R) > 0. 
\end{align*}

Since ${\cal Q}''(z_B)=0$, we deduce that ${\cal Q}''(z)>0$ for $z_R<z<z_B$. That is, ${\cal Q}(z)$ is strictly increasing in $z\in(z_R,z_B)$. It follows from ${\cal Q}'(z_B)=0$ that 
$$
{\cal Q}'(z) > 0 \;\;\mbox{for}\;\;z\in(z_R,z_B). 
$$

Thus, ${\cal Q}(z)$ satisfies 
$$
	\max\left\{\min\left\{{\cal L}{\cal Q}_0(z)+\tilde{u}(z)+h(z) , -\dfrac{d{\cal Q}}{dz}(z) \right\}, J_R(z)-{\cal Q}(z)\right\}= 0. 
$$

   \item[(iii)] the case  $z\in [z_B,\infty).$
   
   Since ${\cal Q}'(z)={\cal Q}''(z)=0$ for $z\ge z_B$, we deduce that 
   	\begin{align*}
   	{\cal L}{\cal Q}(z) + \tilde{u}(z)+h(z)=&{\cal L}{\cal Q}(z_B) + \tilde{u}(z_B)+h(z_B)-\d {\cal Q}(z)+\d {\cal Q}(z_B)+\tilde{u}(z)+h(z)-\tilde{u}(z_B)-h(z_B)\\
   	=&-\d {\cal Q}(z)+\d {\cal Q}(z_B)+\tilde{u}(z)+h(z)-\tilde{u}(z_B)-h(z_B)\\
   	=&\int_{z_B}^z \left(\tilde{u}'(\eta)+h'(\eta)-\d {\cal Q}'(\eta)\right)d\eta \\
   	=&\int_{z_B}^z \left(\hat{\e}(\eta)-\hat{c}(\eta)\right)d\eta>0.
   	\end{align*}

It follows from $J_R'(z) =-{\cal X}(z)<0$ that for $z\ge z_B$
\begin{align*}
J_R(z)-{\cal Q}(z) \le J_R(z_B)-{\cal Q}(z_B)<0,
\end{align*}
where we have used the fact that ${\cal Q}(z)>J_R(z)$ for $z\in(z_R,z_B]$ (see \eqref{eq:inq:Q:J0}).

Thus, we deduce that 
	\begin{align*}
		&\max\left\{\min\left\{{\cal L}{\cal Q}_0(z)+\tilde{u}(z)+h(z) , -\dfrac{d{\cal Q}}{dz}(z) \right\}, J_R(z)-{\cal Q}(z)\right\}\\
		=&\max\left\{0, J_R(z)-{\cal Q}(z)\right\}
		=0. 
	\end{align*}
\end{itemize}

By (i), (ii), and (iii), we conclude that for all $z>0$ ${\cal Q}(z)$ satisfies 
\begin{align*}
\max\left\{\min\left\{{\cal L}{\cal Q}_0(z)+\tilde{u}(z)+h(z) , -\dfrac{d{\cal Q}}{dz}(z) \right\}, J_R(z)-{\cal Q}(z)\right\}=0.
\end{align*}
\bigskip

\noindent{\bf  (b)}  
From (i), (ii), and (iii) in the proof of part (a), we can rewrite the four regions ${\bf WR}$, ${\bf RR}$, ${\bf AR}$, and ${\bf IR}$ as follows: {
\begin{align*}
{\bf RR}&=\{z>0\mid z\le z_R \},\;\;{\bf WR}=\{z>0\mid z>z_R   \},\\
{\bf AR}&=\{z>0\mid z_B \le z \},\;\;{\bf IR}=\{z>0 \mid 0 <z<z_B \}.
\end{align*} }

\noindent{\bf  (c)}  Since ${\cal Q}(z)=J_R(z)$ for $z\in(0,z_R)$, it follows from \eqref{eq:sign-X-prime} and \eqref{eq:J0-prime} in Appendix \ref{sec:after-dual} that 
\begin{equation*}
	{\cal Q}''(z) = J_R''(z) > 0. 	
\end{equation*}
From (ii) in the proof of part (a), 
$$
{\cal Q}''(z) > 0\;\;\;\mbox{for}\;\;z\in(z_R,z_B).
$$

Hence, ${\cal Q}(z)$ is strictly convex in $z\in(0,z_R)$. \\
\bigskip

\noindent{\bf (d)} By construction of ${\cal Q}(z)$ in \eqref{eq:Q-sol-HJBQV}, it is clear that 
$$
\lim_{z\to z_B- } {\cal Q}'(z) = 0. 
$$

On the other hand,  it follows from \eqref{eq:XR-inf}   that 
\begin{align*}
	\lim_{z\to+0}{\cal Q}'(z) =& \lim_{z\to+0} J_R'(z)=-\lim_{z\to+0} {\cal X}_R'(z) = -\infty.
\end{align*}

\end{proof}

We are now ready to state the verification theorem for Problem \ref{pr:dual}. 
\begin{thm}\label{thm:NASH} Let $z_R$ and $z_B$ be as in Proposition \ref{pro:two-free-boundaries}, and consider  strategies $({\cal D}^{z_B}, \tau_{z_R})$ defined by 
	\begin{equation}\label{eq:barriers}
	{\cal D}_t^{z_B}= \min\left\{1, \inf_{0\le s \le t }\dfrac{z_B}{{\cal Y}_s}\right\}\;\;\mbox{for}\;\;t\ge 0,\;\;\;\tau_{z_R}=\inf\{t\ge 0 \mid {\cal Z}_t^{{\cal D}^{z_B}} < z_R \}
\end{equation}
with ${\cal Z}_t^{{\cal D}^{z_B}}= {\cal Y}_t {\cal D}_t^{z_B}\;\;\mbox{for}\;\;t\ge 0.$
	 
For ${\cal Q}(z)$ in \eqref{eq:Q-sol-HJBQV}, we have 
	\begin{itemize}
		\item[(a)]  ${\cal Q}(z)\le {\cal J}_0(z;{\cal D}, \tau_{z_R})$ for any ${\cal D}\in{\cal NI}$.
		\item[(b)] ${\cal Q}(z)\ge {\cal J}_0(z;{\cal D}^{z_B}, \tau)$ for any $\tau\in{\cal S}$
		\item[(c)] ${\cal Q}(z)= {\cal J}_0(z;{\cal D}^{z_B}, \tau_{z_R})$
	\end{itemize}
That is, the pair $(\widehat{\cal D}, \hat{\tau})=({\cal D}^{z_B}, \tau_{z_R})$ is a {\it Nash-equilibrium} for Problem \ref{pr:dual} and ${\cal Q}(z)=J(z)$ is the value function of the game in \eqref{eq:game}. 
\end{thm}
\begin{proof}

	
		\noindent {\bf  (a)}  Define a process  $G$ by 
	\begin{align}
	G_t^{{\cal D}} =&\int_0^t e^{-\d s} \left\{\left(\tilde{u}\left(\frac{{\cal Z}_s^{\cal D}}{\k_1} \right)+\e_1 {\cal Z}_s^{\cal D}\right){\bf 1}_{\{{\cal Z}_s^{\cal D}>z_S\}}+\left(\tilde{u}\left(\frac{{\cal Z}_s^{\cal D}}{\k_2} \right)+\e_2 {\cal Z}_s^{\cal D}\right){\bf 1}_{\{{\cal Z}_s^{\cal D}\le z_S\}}\right\}ds +e^{-\d t}{\cal Q}({\cal Z}_t^{\cal D}).
	\end{align}

Since ${\cal Q}$ is $C^1$ as well as $C^2$ outside a finite set $\{z_R\}$, we can still apply  generalized It\^{o}'s lemma (see Proposition 9 in \citet{Harrison}, Exercise 6.24 in \citet{KS2}).

It follows that 
\begin{align}
\label{eq:ito}
&d G_t^{{\cal D}}\\ =&e^{-\d t}\left\{\left(\tilde{u}\left(\frac{{\cal Z}_t^{\cal D}}{\k_1} \right)+\e_1 {\cal Z}_t^{\cal D}\right){\bf 1}_{\{{\cal Z}_t^{\cal D}>z_S\}}+\left(\tilde{u}\left(\frac{{\cal Z}_t^{\cal D}}{\k_2} \right)+\e_2 {\cal Z}_t^{\cal D}\right){\bf 1}_{\{{\cal Z}_t^{\cal D}\le z_S\}}\right\}dt+e^{-\d t}d{\cal Q}({\cal Z}_t^{\cal D}) -\d {\cal Q}({\cal Z}_t^{\cal D})dt\nonumber\\
=&e^{-\d t}\left\{\left(\tilde{u}\left(\frac{{\cal Z}_t^{\cal D}}{\k_1} \right)+\e_1 {\cal Z}_t^{\cal D}\right){\bf 1}_{\{{\cal Z}_t^{\cal D}>z_S\}}+\left(\tilde{u}\left(\frac{{\cal Z}_t^{\cal D}}{\k_2} \right)+\e_2 {\cal Z}_t^{\cal D}\right){\bf 1}_{\{{\cal Z}_t^{\cal D}\le z_S\}}\right\}dt-\d {\cal Q}({\cal Z}_t^{\cal D})dt+e^{-\d t}{\cal Q}'({\cal Z}_t^{\cal D}) d{\cal Z}_t^{\cal D}\nonumber\\
+&e^{-\d t}\dfrac{1}{2}{\cal Q}''({\cal Z}_t^{\cal D})(d{\cal Z}_t^{\cal D})^2 + e^{-\d t}{\cal Q}' ({\cal Z}_t^{\cal D}) d{\cal D}_t\nonumber\\
=&e^{-\d t}\left\{\frac{\t^2}{2}({\cal Z}_t^{\cal D})^2{\cal Q}''({\cal Z}_t^{\cal D})+(\d-r){\cal Z}_t^{\cal D}{\cal Q}'({\cal Z}_t^{\cal D})-\d {\cal Q}({\cal Z}_t^{\cal D})+\left(\tilde{u}\left(\frac{{\cal Z}_t^{\cal D}}{\k_1} \right)+\e_1 {\cal Z}_t^{\cal D}\right){\bf 1}_{\{{\cal Z}_t^{\cal D}>z_S\}}\right.\nonumber\\+&\left.\left(\tilde{u}\left(\frac{{\cal Z}_t^{\cal D}}{\k_2} \right)+\e_2 {\cal Z}_t^{\cal D}\right){\bf 1}_{\{{\cal Z}_t^{\cal D}\le z_S\}}\right\}dt
+e^{-\d t}{\cal Q}'({\cal Z}_t^{\cal D})d{\cal D}_t^{c} +e^{-\d t}\left({\cal Q}({\cal Z}_t^{\cal D})-{\cal Q}({\cal Z}_{t-}^{\cal D})\right)-e^{-\d t}\t{\cal Z}_t^{\cal D}{\cal Q}'({\cal Z}_t^{\cal D})dW_t,\nonumber
\end{align}
where $\D {\cal D}_t \equiv {\cal D}_t - {\cal D}_{t-}$ and  ${\cal D}^c$ is the continuous part of ${\cal D}$. 

Let us denote ${\tau}_n$ by ${\tau}_n=\inf\{t\ge 0 \mid {\cal Z}_t^{\cal D} > n \}\wedge \tau_{z_R}$ for each $n\in\mathbb{N}$. Then, we have
	\begin{align*}
	&G_{{\tau}_n\wedge t}^{{\cal D}} = {\cal Q}(z) +\underbrace{\int_0^{{\tau}_n\wedge t}e^{-\d s}{\cal Q}'({\cal Z}_s^{\cal D})d{\cal D}_s^c}_{{\bf (A)}}+\underbrace{\sum_{s=0}^{{\tau}_n\wedge t}e^{-\d s}\left({\cal Q}({\cal Z}_s^{\cal D})-{\cal Q}({\cal Z}_{s-}^{\cal D})\right)}_{{\bf (B)}}+\underbrace{\int_0^{{\tau}_n\wedge t}e^{-\d s}(-\t){\cal Z}_s^{\cal D}{\cal Q}'({\cal Z}_s^{\cal D})dW_s}_{{\bf (C)}}\\+&\underbrace{\int_0^{{\tau}_n\wedge t} e^{-\d s}\left\{{\cal L} {\cal Q}({\cal Z}_s^{\cal D})+\left(\tilde{u}\left(\frac{{\cal Z}_s^{\cal D}}{\k_1} \right)+\e_1 {\cal Z}_s^{\cal D}\right){\bf 1}_{\{{\cal Z}_s^{\cal D}>z_S\}}+\left(\tilde{u}\left(\frac{{\cal Z}_s^{\cal D}}{\k_2} \right)+\e_2 {\cal Z}_s^{\cal D}\right){\bf 1}_{\{{\cal Z}_s^{\cal D}\le z_S\}}\right\}ds}_{{\bf (D)}}.
	\end{align*}
Taking expectations, we have 
	\begin{align*}
	&\mathbb{E}\left[\int_0^{{\tau}_n\wedge t} e^{-\d s} \left\{\left(\tilde{u}\left(\frac{{\cal Z}_s^{\cal D}}{\k_1} \right)+\e_1 {\cal Z}_s^{\cal D}\right){\bf 1}_{\{{\cal Z}_s^{\cal D}>z_S\}}+\left(\tilde{u}\left(\frac{{\cal Z}_s^{\cal D}}{\k_2} \right)+\e_2 {\cal Z}_s^{\cal D}\right){\bf 1}_{\{{\cal Z}_s^{\cal D}\le z_S\}}\right\}ds +e^{-\d ({\tau}_n \wedge t)}{\cal Q}({\cal Z}_{({{\tau}_n}\wedge t)-}^{\cal D})\right]\\
	=&{\cal Q}(z)+\mathbb{E}[{\bf (A)}]+\mathbb{E}[{\bf (B)}]+\mathbb{E}[{\bf (C)}]+\mathbb{E}[{\bf (D)}].
	\end{align*}
Let us denote ${\cal M}_t$ by 
\begin{equation}
	{\cal M}_t  = \int_0^t e^{-\d s}(-\t){\cal Z}_s^{\cal D}{\cal Q}'({\cal Z}_s^{\cal D})dW_s\;\;\;\mbox{for}\;\;t\ge0.
\end{equation}

Since ${\cal Q}$ satisfies  HJBQV \eqref{eq:HJBQV} and ${\cal M}_{{\tau}_n\wedge t}$ is a martingale, we have
\begin{equation}\label{eq:inequality1}
	\mathbb{E}\left[{\bf (A)}\right]\ge 0,\;\;\mathbb{E}\left[{\bf (C)}\right]=0,\;\;\mbox{and}\;\;\mathbb{E}\left[{\bf (D)}\right] \ge 0. 
\end{equation}

Note that 
\begin{equation}
	{\cal Q}({\cal Z}_s^{\cal D})-{\cal Q}({\cal Z}_{s-}^{\cal D})=\int_{{\cal Z}_s^{\cal D}-\D {\cal Z}_s^{\cal D}}^{{\cal Z}_s^{\cal D}}{\cal Q}_z(\nu)d\nu \ge 0, 
\end{equation}
where $\D {\cal Z}_s^{\cal D} = {\cal Y}_s \D {\cal D}_s \le 0$. 

It follows that 
\begin{equation}\label{eq:inequality2}
	\mathbb{E}\left[{\bf (B)}\right] \ge 0.
\end{equation}

Hence, we deduce that 
\begin{align}
\label{eq:1-1}
{\cal Q}(z)\le &	\mathbb{E}\left[\int_0^{{\tau}_n\wedge t} e^{-\d s} \left\{\left(\tilde{u}\left(\frac{{\cal Z}_s^{\cal D}}{\k_1} \right)+\e_1 {\cal Z}_s^{\cal D}\right){\bf 1}_{\{{\cal Z}_s^{\cal D}>z_S\}}+\left(\tilde{u}\left(\frac{{\cal Z}_s^{\cal D}}{\k_2} \right)+\e_2 {\cal Z}_s^{\cal D}\right){\bf 1}_{\{{\cal Z}_s^{\cal D}\le z_S\}}\right\}ds\right.\\ +&\left.e^{-\d ({\tau}_n \wedge t)}{\cal Q}({\cal Z}_{({{\tau}_n}\wedge t)-}^{\cal D})\right] .\nonumber
\end{align}

Let us temporarily denote $\varphi(z)$ by 
	\begin{align*}
		\varphi(z):=&h(z)+\tilde{u}(z)   =\left(\tilde{u}\left(\frac{z}{\k_1} \right)+\e_1 z\right){\bf 1}_{\{z>z_S\}}+\left(\tilde{u}\left(\frac{z}{\k_2} \right)+\e_2 z\right){\bf 1}_{\{z\le z_S\}}.
	\end{align*}
For $s\in[0,\tau_n\wedge t)$, we have  $$z_R \le {\cal Z}_s^{\cal D} \le {\cal Y}_s.$$ 

This leads to
	\begin{align*}
	\left(\varphi({\cal Z}_s^{\cal D})\right)_+=& \left(\left(\tilde{u}\left(\frac{{\cal Z}_s^{\cal D}}{\k_1} \right)+\e_1 {\cal Z}_s^{\cal D}\right){\bf 1}_{\{{\cal Z}_s^{\cal D}>z_S\}}+\left(\tilde{u}\left(\frac{{\cal Z}_s^{\cal D}}{\k_2} \right)+\e_2 {\cal Z}_s^{\cal D}\right){\bf 1}_{\{{\cal Z}_s^{\cal D}\le z_S\}}\right)_+\\
		\le&\left(\tilde{u}\left(\frac{z_R}{\k_2} \right)+\e_1 {\cal Y}_s\right)_+\le \left(\tilde{u}\left(\frac{z_R}{\k_2} \right)\right)_+ +\e_1 {\cal Y}_s.
	\end{align*}
Thus, we deduce that 
\begin{align*}
			&\mathbb{E}\left[\int_0^{{\tau}_n\wedge t} e^{-\d s} 	\left(\varphi({\cal Z}_s^{\cal D})\right)_+ ds\right]
			\le\mathbb{E}\left[\int_0^{\infty} e^{-\d s} \left\{\left(\tilde{u}\left(\frac{z_R}{\k_2} \right)\right)_+ +\e_1 {\cal Y}_s\right\}ds\right]<\infty. 
\end{align*}
The {\it dominated convergence theorem} implies that 
\begin{equation}\label{eq:DCT}
	\lim_{t\uparrow \infty}\lim_{n\uparrow \infty} \mathbb{E}\left[\int_0^{{\tau}_n\wedge t} e^{-\d s} 	\left(\varphi({\cal Z}_s^{\cal D})\right)_+ ds\right]=\mathbb{E}\left[\int_0^{\tau_{z_R}} e^{-\d s} 	\left(\varphi({\cal Z}_s^{\cal D})\right)_+ ds\right]<\infty.
\end{equation}
Moreover, the {\it monotone convergence theorem} implies that 
\begin{equation}\label{eq:MCT}
	\lim_{t\uparrow \infty}\lim_{n\uparrow \infty} \mathbb{E}\left[\int_0^{{\tau}_n\wedge t} e^{-\d s} 	\left(\varphi({\cal Z}_s^{\cal D})\right)_- ds\right]=\mathbb{E}\left[\int_0^{\tau_{z_R}} e^{-\d s} 	\left(\varphi({\cal Z}_s^{\cal D})\right)_- ds\right]
\end{equation}
It follows from \eqref{eq:DCT} and \eqref{eq:MCT} that 
\begin{align}
\label{eq:1-2}
&\lim_{t\uparrow \infty}\lim_{n\uparrow \infty}\mathbb{E}\left[\int_0^{{\tau}_n\wedge t} e^{-\d s} \left\{\left(\tilde{u}\left(\frac{{\cal Z}_s^{\cal D}}{\k_1} \right)+\e_1 {\cal Z}_s^{\cal D}\right){\bf 1}_{\{{\cal Z}_s^{\cal D}>z_S\}}+\left(\tilde{u}\left(\frac{{\cal Z}_s^{\cal D}}{\k_2} \right)+\e_2 {\cal Z}_s^{\cal D}\right){\bf 1}_{\{{\cal Z}_s^{\cal D}\le z_S\}}\right\}ds \right]\\
=&\mathbb{E}\left[\int_0^{\tau_{z_R}} e^{-\d s} \left\{\left(\tilde{u}\left(\frac{{\cal Z}_s^{\cal D}}{\k_1} \right)+\e_1 {\cal Z}_s^{\cal D}\right){\bf 1}_{\{{\cal Z}_s^{\cal D}>z_S\}}+\left(\tilde{u}\left(\frac{{\cal Z}_s^{\cal D}}{\k_2} \right)+\e_2 {\cal Z}_s^{\cal D}\right){\bf 1}_{\{{\cal Z}_s^{\cal D}\le z_S\}}\right\}ds \right]. \nonumber
\end{align}

Since ${\cal Q}'(z) \le 0$, we have
\begin{equation*}
	e^{-\d (\tau_n \wedge t )}\left({\cal Q}({\cal Z}_{(\tau_n \wedge t)-}^{\cal D})\right)_+ \le \left({\cal Q}(z_R)\right)_+ <\infty,
\end{equation*}
where we have used fact that $z_R \le {\cal Z}_{(\tau_n\wedge t)-}^{\cal D} \;\;\;\mbox{for all}\;\;\;t\ge 0. $

Thus,  Fatou's lemma implies that 
\begin{equation}\label{eq:1-3}
	\limsup_{n\uparrow \infty}\mathbb{E}\left[e^{-\d(\tau_n \wedge t)}{\cal Q}({\cal Z}_{(\tau_n\wedge t)-}^{\cal D})\right] \le \mathbb{E}\left[ e^{-\d (\tau_{z_R}\wedge t)}{\cal Q}({\cal Z}_{(\tau_{z_R}\wedge t) -}^{\cal D})\right].
\end{equation}

Note that 
\begin{equation}
	e^{-\d (\tau_{z_R}\wedge t)}{\cal Q}({\cal Z}_{(\tau_{z_R}\wedge t)-}^{\cal D})=e^{-\d t}{\cal Q}({\cal Z}_{t-}^{\cal D}){\bf 1}_{\{t<\tau_{z_R} \}} + e^{-\d \tau_{z_R}}J_R({\cal Z}_{\tau_{z_R}-}^{\cal D}){\bf 1}_{\{t\ge \tau_{z_R} \}}.
\end{equation}

Since ${\cal Q}({\cal Z}_t^{\cal D})\le {\cal Q}(z_R)$ for $t\in[0,\tau_{z_R})$ and $J_R({\cal Z}_{\tau_{z_R}-}^{\cal D})\le J_R(z_R)$, it follows from Fatou's lemma that 
\begin{align}
\label{eq:1-4}
&\limsup_{t\uparrow \infty}\mathbb{E}\left[e^{-\d (\tau_{z_R}\wedge t)}{\cal Q}({\cal Z}_{(\tau_{z_R}\wedge t)- }^{\cal D})\right]\\\le& \limsup_{t\uparrow \infty}\mathbb{E}\left[e^{-\d t}{\cal Q}({\cal Z}_{t-}^{\cal D}){\bf 1}_{\{t<\tau_{z_R} \}} \right]+\limsup_{t\uparrow\infty}\mathbb{E}\left[ e^{-\d \tau_{z_R}}J_R({\cal Z}_{\tau_{z_R}}){\bf 1}_{\{t\ge \tau_{z_R} \}}\right]\nonumber\\
\le&\limsup_{t\uparrow \infty}\mathbb{E}\left[e^{-\d t}{\cal Q}(z_R){\bf 1}_{\{t<\tau_{z_R} \}} \right]+\limsup_{t\uparrow\infty}\mathbb{E}\left[ e^{-\d \tau_{z_R}}J_R({\cal Z}_{\tau_{z_R}}){\bf 1}_{\{t\ge \tau_{z_R} \}}\right]\nonumber\\
\le&\mathbb{E}\left[ e^{-\d \tau_{z_R}}J_R({\cal Z}_{\tau_{z_R}}){\bf 1}_{\{\tau_{z_R}<\infty \}}\right].\nonumber
\end{align}

By \eqref{eq:1-1}, \eqref{eq:1-2}, \eqref{eq:1-3}, and \eqref{eq:1-4}, we conclude that 
\begin{equation}
{\cal Q}(z) \le {\cal J}_0 ( z; {\cal D}, \tau_{z_R}) 
\end{equation}
for any ${\cal D}\in{\cal NI}$. \\

\noindent{\bf (b)}  Let ${\cal D}_t^{z_B}$ be as in the statement of theorem.  For each $n\in\mathbb{N}$ define the stopping time $\tau_n$ as 
$$
\tau_n:= \inf\left\{t>0\mid {\cal Z}_t^{{\cal D}^{z_B}} < \frac{1}{n} \right\} \wedge \tau 
$$
for an arbitrary stopping time $\tau\in{\cal S}$.

From the definition of ${\cal D}^{z_B}$, it is easy to see that ${\cal D}_t^{z_B}$ is continuous, 
\begin{equation}
	\int_0^t e^{-\d s} {\cal Q}'({\cal Z}_s^{{\cal D}^{z_B}})d{\cal D}_s^{z_B} =0 \;\;\;\mbox{for}\;\;t\ge 0,
\end{equation}
and 
	\begin{eqnarray}
		\begin{split}
			\begin{cases}
				&\dfrac{\t^2}{2}({\cal Z}_t^{{\cal D}^{z_B}})^2 {\cal Q}''({\cal Z}_t^{{\cal D}^{z_B}})+(\d-r){\cal Z}_t^{{\cal D}^{z_B}}{\cal Q}'({\cal Z}_t^{{\cal D}^{z_B}})-\d {\cal Q}({\cal Z}_t^{{\cal D}^{z_B}}) + \tilde{u}({\cal Z}_t^{{\cal D}^{z_B}}) + h({\cal Z}_t^{{\cal D}^{z_B}})=0 \vspace{2mm}\\&\mbox{for all}\;\;t\in[0,\tau_{z_R}), \vspace{2mm}\\
				&\dfrac{\t^2}{2}({\cal Z}_t^{{\cal D}^{z_B}})^2 {\cal Q}''({\cal Z}_t^{{\cal D}^{z_B}})+(\d-r){\cal Z}_t^{{\cal D}^{z_B}}{\cal Q}'({\cal Z}_t^{{\cal D}^{z_B}})-\d {\cal Q}({\cal Z}_t^{{\cal D}^{z_B}}) + \tilde{u}({\cal Z}_t^{{\cal D}^{z_B}}) + h({\cal Z}_t^{{\cal D}^{z_B}})\le 0 \vspace{2mm}\\&\mbox{for all}\;\;t\in(\tau_{z_R},\infty).
			\end{cases}
		\end{split}
	\end{eqnarray}
 
This leads to
\begin{equation}\label{eq:inequality}
	\mathbb{E}\left[{\bf (A)}\right]=\mathbb{E}\left[{\bf (B)}\right]=\mathbb{E}\left[{\bf (C)}\right]=0\;\;\;\mbox{and}\;\;\;\mathbb{E}\left[{\bf (D)}\right]\le 0
\end{equation}
for ${\cal D}={\cal D}^{z_B}$ in \eqref{eq:ito}.

By utilizing the above localizing argument, it follow from \eqref{eq:inequality} that 
\begin{align}
	\label{eq:2-1}
{\cal Q}(z)\ge&\mathbb{E}\left[\int_0^{{\tau}_n\wedge t} e^{-\d s} \left\{\left(\tilde{u}\left(\frac{{\cal Z}_s^{{\cal D}^{z_B}}}{\k_1} \right)+\e_1 {\cal Z}_s^{{\cal D}^{z_B}}\right){\bf 1}_{\{{\cal Z}_s^{{\cal D}^{z_B}}>z_S\}}+\left(\tilde{u}\left(\frac{{\cal Z}_s^{{\cal D}^{z_B}}}{\k_2} \right)+\e_2 {\cal Z}_s^{{\cal D}^{z_B}}\right){\bf 1}_{\{{\cal Z}_s^{{\cal D}^{z_B}}\le z_S\}}\right\}ds\right.\\ +&\left.e^{-\d ({\tau}_n \wedge t)}{\cal Q}({\cal Z}_{{{\tau}_n}\wedge t}^{{\cal D}^{z_B}})\right].\nonumber
\end{align}

Note that ${\cal Z}_t^{{\cal D}^{z_B}} \le z_B$ for all $t\ge 0$. Thus, similarly to the derivation above in the part (a),  we have 
\begin{align}
\label{eq:2-2}
\lim_{t\uparrow \infty}\lim_{n\uparrow \infty}\;&\mathbb{E}\left[\int_0^{{\tau}_n\wedge t} e^{-\d s} \left\{\left(\tilde{u}\left(\frac{{\cal Z}_s^{{\cal D}^{z_B}}}{\k_1} \right)+\e_1 {\cal Z}_s^{{\cal D}^{z_B}}\right){\bf 1}_{\{{\cal Z}_s^{{\cal D}^{z_B}}>z_S\}}+\left(\tilde{u}\left(\frac{{\cal Z}_s^{{\cal D}^{z_B}}}{\k_2} \right)+\e_2 {\cal Z}_s^{{\cal D}^{z_B}}\right){\bf 1}_{\{{\cal Z}_s^{{\cal D}^{z_B}}\le z_S\}}\right\}ds\right]\;\;\\
=&\mathbb{E}\left[\int_0^{{\tau}} e^{-\d s} \left\{\left(\tilde{u}\left(\frac{{\cal Z}_s^{{\cal D}^{z_B}}}{\k_1} \right)+\e_1 {\cal Z}_s^{{\cal D}^{z_B}}\right){\bf 1}_{\{{\cal Z}_s^{{\cal D}^{z_B}}>z_S\}}+\left(\tilde{u}\left(\frac{{\cal Z}_s^{{\cal D}^{z_B}}}{\k_2} \right)+\e_2 {\cal Z}_s^{{\cal D}^{z_B}}\right){\bf 1}_{\{{\cal Z}_s^{{\cal D}^{z_B}}\le z_S\}}\right\}ds\right].\nonumber
\end{align}

Since ${\cal Q}({\cal Z}_{\tau_n\wedge t}^{{\cal D}^{z_B}}) \ge {\cal Q}(z_B)$ for all $t \ge 0$, Fatou's lemma implies that 
\begin{align*}
			\liminf_{t\to\infty}\liminf_{n\to\infty} \;\mathbb{E}\left[e^{-\d(\tau_n\wedge t)}{\cal Q}({\cal Z}_{\tau_n\wedge t})\right]&\ge \mathbb{E}\left[e^{-\d \tau}{\cal Q}({\cal Z}_{\tau}^{{\cal D}^{z_B}}){\bf 1}_{\{\tau < \infty \}}\right] +\liminf_{t\to\infty} \mathbb{E}\left[e^{-\d t} {\cal Q}({\cal Z}_t^{{\cal D}^{z_B}})\right] \\
	&\ge  \mathbb{E}\left[e^{-\d \tau}{\cal Q}({\cal Z}_{\tau}^{{\cal D}^{z_B}}){\bf 1}_{\{\tau < \infty \}}\right].
\end{align*}
Since ${\cal Q}(z) \ge J_R(z)$ for all $z>0$ in \eqref{eq:HJBQV}, we have 
\begin{equation}\label{eq:2-3}
		\liminf_{t\to\infty}\liminf_{n\to\infty} \;\mathbb{E}\left[e^{-\d(\tau_n\wedge t)}{\cal Q}({\cal Z}_{\tau_n\wedge t})\right]\ge  \mathbb{E}\left[e^{-\d \tau}J_R({\cal Z}_{\tau}^{{\cal D}^{z_B}}){\bf 1}_{\{\tau < \infty \}}\right].
\end{equation}

It follows from \eqref{eq:2-1}, \eqref{eq:2-2}, and \eqref{eq:2-3} that 
\begin{equation*}
	{\cal Q}(z) \ge {\cal J}_0(z; {\cal D}^{z_B}, \tau)
\end{equation*}
for any $\tau\in{\cal S}$.\\

\noindent {\bf (c)} 	For ${\cal D}={\cal D}^{z_B}$ and $\tau_n={\tau_R}$, 
\begin{equation}\label{eq:range_Z}
	z_R\le {\cal Z}_t^{{\cal D}^{z_B}}\le z_B \;\;\;t\in[0,\tau_R].
\end{equation}
By using the arguments in  parts (a) and (b) with \eqref{eq:range_Z}, it is easy to confirm that 
$$
\mathbb{E}[{\bf (A)}]=\mathbb{E}[{\bf (B)}]=\mathbb{E}[{\bf (C)}]=\mathbb{E}[{\bf (D)}]=0\;\;\mbox{in}\;\;\eqref{eq:ito}.
$$
That is, 
\begin{align*}
{\cal Q}(z)=&\mathbb{E}\left[\int_0^{{\tau}_{z_R}\wedge t} e^{-\d s} \left\{\left(\tilde{u}\left(\frac{{\cal Z}_s^{\cal D}}{\k_1} \right)+\e_1 {\cal Z}_s^{\cal D}\right){\bf 1}_{\{{\cal Z}_s^{\cal D}>z_S\}}+\left(\tilde{u}\left(\frac{{\cal Z}_s^{\cal D}}{\k_2} \right)+\e_2 {\cal Z}_s^{\cal D}\right){\bf 1}_{\{{\cal Z}_s^{\cal D}\le z_S\}}\right\}ds\right.\\ +&\left.e^{-\d ({\tau}_{z_R} \wedge t)}{\cal Q}({\cal Z}_{({{\tau}_{z_R}}\wedge t)}^{\cal D})\right].
\end{align*}

From \eqref{eq:range_Z}, we can easily get 
	\begin{align*}
		\mathbb{E}\left[\int_0^{{\tau}_{z_R}} e^{-\d s} \left|\left\{\left(\tilde{u}\left(\frac{{\cal Z}_s^{\cal D}}{\k_1} \right)+\e_1 {\cal Z}_s^{\cal D}\right){\bf 1}_{\{{\cal Z}_s^{\cal D}>z_S\}}+\left(\tilde{u}\left(\frac{{\cal Z}_s^{\cal D}}{\k_2} \right)+\e_2 {\cal Z}_s^{\cal D}\right){\bf 1}_{\{{\cal Z}_s^{\cal D}\le z_S\}}\right\}\right|ds\right] <\infty
	\end{align*}
and 
\begin{align*}
	\mathbb{E}\left[e^{-\d \tau_{z_R}} \left|J_R({\cal Z}_{\tau_{z_R}}^{{\cal D}^{z_B}})\right|\right]<\infty.
\end{align*} 

The {\it dominated convergence theorem} implies that 
	\begin{align*}
		{\cal Q}(z)=&\lim_{t\uparrow \infty}\mathbb{E}\left[\int_0^{{\tau}_{z_R}\wedge t} e^{-\d s} \left\{\left(\tilde{u}\left(\frac{{\cal Z}_s^{\cal D}}{\k_1} \right)+\e_1 {\cal Z}_s^{\cal D}\right){\bf 1}_{\{{\cal Z}_s^{\cal D}>z_S\}}+\left(\tilde{u}\left(\frac{{\cal Z}_s^{\cal D}}{\k_2} \right)+\e_2 {\cal Z}_s^{\cal D}\right){\bf 1}_{\{{\cal Z}_s^{\cal D}\le z_S\}}\right\}ds\right.\\ +&\left.e^{-\d ({\tau}_{z_R} \wedge t)}{\cal Q}({\cal Z}_{({{\tau}_{z_R}}\wedge t)}^{\cal D})\right]\\
		=&\mathbb{E}\left[\int_0^{{\tau}_{z_R}} e^{-\d s} \left\{\left(\tilde{u}\left(\frac{{\cal Z}_s^{\cal D}}{\k_1} \right)+\e_1 {\cal Z}_s^{\cal D}\right){\bf 1}_{\{{\cal Z}_s^{\cal D}>z_S\}}+\left(\tilde{u}\left(\frac{{\cal Z}_s^{\cal D}}{\k_2} \right)+\e_2 {\cal Z}_s^{\cal D}\right){\bf 1}_{\{{\cal Z}_s^{\cal D}\le z_S\}}\right\}ds \right.\\ +&\left.{\bf 1}_{\{\tau_{z_R}<\infty \}}e^{-\d {\tau}_{z_R}}J_R({\cal Z}_{{{\tau}_{z_R}}}^{\cal D})\right].
	\end{align*}
\end{proof}

\subsection{Verification II: Duality Theorem}

 Since  the pair $({\cal D}^{z_B}, \tau_{z_R})$ is a {\it Nash-equilibrium} of Problem \ref{pr:dual} (Theorem \ref{thm:NASH}), the dual value function $J(z)$ is given by 
\begin{equation*}
	\begin{split}
		J(y) ={\cal J}_0(y; {\cal D}^{z_B},\tau_{z_R})={\cal Q}(y).
	\end{split}
\end{equation*}

Hence, we can write the dual value function $J(y)$ in the explicit form:
\begin{eqnarray}
	\begin{split}\label{eq:J-sol-HJBQV}
		{ J}(y)=
		\begin{cases}
			J_R(y)\;\;\;&\mbox{for}\;\;y\le z_R,\\
			E_1 y^{n_1} + E_2 y^{n_2} + \phi_h(y)+J_R(y)\;\;\;&\mbox{for}\;\;z_R \le y \le z_B,\\
			E_1 y_B^{n_1} + E_2 y_B^{n_2} + \phi_h(z_B)+J_R(z_B)\;\;\;&\mbox{for}\;\;\;z_B \le  y,
		\end{cases}
	\end{split}
\end{eqnarray}
where 
	\begin{equation*}
		E_1 =-\dfrac{2}{\t^2(n_1-n_2)}\int_{z_R}^\infty \eta^{-n_1-1}h(\eta)d\eta\;\;\;\mbox{and}\;\;\;E_2 =-\dfrac{2}{\t^2(n_1-n_2)}\int_0^{z_R} \eta^{-n_2-1}h(\eta)d\eta.
	\end{equation*}

\begin{lem}\label{eq:equality-budet}
	For given $y>0$ and the Nash-equilibrium $({\cal D}^{z_B},\tau_{z_R})$ of Problem \ref{pr:dual}, the following relationship holds:
		\begin{eqnarray}
			\begin{split}
			-J'(y)=&\mathbb{E}\left[\int_0^{\tau_R}{\cal H}_t\left(\hat{c}({\cal Z}_t^{{\cal D}^{z_B}})-\hat{\e}({\cal Z}_t^{{\cal D}^{z_B}})\right)dt +{\bf 1}_{\{\tau_{z_R}<\infty \}}{\cal H}_{\tau_R}{\cal X}_{R}({\cal Z}_{\tau_R}^{{\cal D}^{z_B}}) \right]\\
			=&\mathbb{E}\left[\int_0^{\tau_R}{\cal H}_t{\cal D}_t^{z_B}\left(\hat{c}({\cal Z}_t^{{\cal D}^{z_B}})-\hat{\e}({\cal Z}_t^{{\cal D}^{z_B}})\right)dt +{\bf 1}_{\{\tau_{z_R}<\infty \}}{\cal H}_{\tau_R}{\cal D}_{\tau_R}^{z_B}{\cal X}_{R}({\cal Z}_{\tau_R}^{{\cal D}^{z_B}}) \right].
			\end{split}
		\end{eqnarray}
Recall that  ${\cal X}_R(y)=-J_R'(y)$ (see Appendix \ref{sec:after-dual}), 
\begin{footnotesize}
	\begin{align*}
		\hat{c}(z) =& \frac{1}{\k_1}I\left(\frac{z}{\k_1}\right){\bf 1}_{\{z>z_S \}} + \frac{1}{\k_2}I\left(\frac{z}{\k_2}\right){\bf 1}_{\{z\le z_S \}},\;\;\;\mbox{and}\;\;\;
		\hat{\e}(z)=\e_1 {\bf 1}_{\{z>z_S \}} + \e_2{\bf 1}_{\{z\le z_S \}}.
	\end{align*}
\end{footnotesize}
\end{lem}
\begin{proof}

First, we will show that 	
\begin{equation*}
-J'(y)=\mathbb{E}\left[\int_0^{\tau_R}{\cal H}_t{\cal D}_t^{z_B}\left(\hat{c}({\cal Z}_t^{{\cal D}^{z_B}})-\hat{\e}({\cal Z}_t^{{\cal D}^{z_B}})\right)dt +{\bf 1}_{\{\tau_{z_R}<\infty \}}{\cal H}_{\tau_R}{\cal D}_{\tau_R}^{z_B}{\cal X}_{R}({\cal Z}_{\tau_R}^{{\cal D}^{z_B}}) \right].
\end{equation*}
	Let us temporarily denote ${\cal Q}_1(y)$ by 
	\begin{equation*}
		{\cal Q}_1(y) = - yJ'(y).
	\end{equation*}

Define a process ${\cal N}$ by 
\begin{equation}
	{\cal N}_t=\int_0^t e^{-\d s}\left({\cal Z}_s^{{\cal D}^{z_B}} \hat{c}({\cal Z}_s^{{\cal D}^{z_B}})-{\cal Z}_s^{{\cal D}^{z_B}} \hat{\e}({\cal Z}_s^{{\cal D}^{z_B}})\right)ds +e^{-\d t }{\cal Q}_1({\cal Z}_t^{{\cal D}^{z_B}}).
\end{equation}

Since ${\cal X}_R(y)=-J_R'(y)$ is $C^2$, it is easy to confirm that ${\cal Q}_1$ is $C^2$ outside $\{z_R, z_B\}$.  Thus, we can apply  generalized It\^{o}'s lemma to ${\cal Q}_1({\cal Z}_t^{{\cal D}^{z_B}})$. Thus, we have 
	\begin{align*}
		d{\cal N}_t =& e^{-\d t}\left({\cal Z}_t^{{\cal D}^{z_B}} \hat{c}({\cal Z}_t^{{\cal D}^{z_B}})-{\cal Z}_t^{{\cal D}^{z_B}} \hat{\e}({\cal Z}_t^{{\cal D}^{z_B}})\right)dt+e^{-\d t}d {\cal Q}_1({\cal Z}_t^{{\cal D}^{z_B}})-\d e^{-\d t}{\cal Q}_1({\cal Z}_t^{{\cal D}^{z_B}})dt\\
	=&e^{-\d t}\left({\cal L}{\cal Q}_1({\cal Z}_t^{{\cal D}^{z_B}})+{\cal Z}_t^{{\cal D}^{z_B}} \hat{c}({\cal Z}_t^{{\cal D}^{z_B}})-{\cal Z}_t^{{\cal D}^{z_B}} \hat{\e}({\cal Z}_t^{{\cal D}^{z_B}})\right)dt+e^{-\d t}(-\t){\cal Z}_t^{{\cal D}^{z_B}}{\cal Q}_1'({\cal Z}_t^{{\cal D}^{z_B}})dW_t +{\cal Q}_1'({\cal Z}_t^{{\cal D}^{z_B}})d{\cal D}_t^{z_B}.
	\end{align*}
From this, 
	\begin{align}\label{eq:N1}
		{\cal N}_{\tau_{z_R}\wedge t}^1 =&{\cal Q}_1(z)+\int_0^{\tau_R\wedge t}\left({\cal L}{\cal Q}_1({\cal Z}_s^{{\cal D}^{z_B}})+{\cal Z}_s^{{\cal D}^{z_B}} \hat{c}({\cal Z}_s^{{\cal D}^{z_B}})-{\cal Z}_s^{{\cal D}^{z_B}} \hat{\e}({\cal Z}_s^{{\cal D}^{z_B}})\right)ds+{\cal M}_{\tau_{z_R}\wedge t}^1 \\+&\int_0^{\tau_{z_R}\wedge t}{\cal Q}_1'({\cal Z}_s^{{\cal D}^{z_B}})d{\cal D}_t^{z_B},\nonumber
	\end{align}
where 
$$
{\cal M}_t^1 := \int_0^t e^{-\d s}(-\t){\cal Z}_s^{{\cal D}^{z_B}}{\cal Q}_1'({\cal Z}_s^{{\cal D}^{z_B}})dW_s
$$

By  construction of $J(y)$ and the definition of ${\cal D}^{z_B}$, we deduce that 
\begin{equation*}
	\int_0^{\tau_{z_R}\wedge t}{\cal Q}_1'({\cal Z}_s^{{\cal D}^{z_B}})d{\cal D}_t^{z_B}=-\int_0^{\tau_{z_R}\wedge t}{\cal Z}_s^{{\cal D}^{z_B}} J''({\cal Z}_s^{{\cal D}^{z_B}})d{\cal D}_t^{z_B}=0.
\end{equation*}

Since ${\cal L}J({\cal Z}_s^{{\cal D}^{z_B}})+h({\cal Z}_s^{{\cal D}^{z_B}})+\tilde{u}({\cal Z}_s^{{\cal D}^{z_B}})=0$ for all $t\in[0,\tau_{z_R})$, we have 
$$
{\cal L}{\cal Q}_1({\cal Z}_s^{{\cal D}^{z_B}})+{\cal Z}_s^{{\cal D}^{z_B}} \hat{c}({\cal Z}_s^{{\cal D}^{z_B}})-{\cal Z}_s^{{\cal D}^{z_B}} \hat{\e}({\cal Z}_s^{{\cal D}^{z_B}})=0
$$
for  all $t\in[0,\tau_{z_R})$.

Moreover, it follows from $z_R \le {\cal Z}_t^{{\cal D}^{z_B}} \le z_B$ for $t\in[0,\tau_{z_R})$ that ${\cal M}_{\tau_{z_R}\wedge t}^1$ is a martingale. 

By taking expectation to the both sides of \eqref{eq:N1}, we derive that 
\begin{equation}
	{\cal Q}_1(z) =\mathbb{E}\left[\int_0^{\tau_{z_R}\wedge t}e^{-\d s}\left({\cal Z}_s^{{\cal D}^{z_B}} \hat{c}({\cal Z}_s^{{\cal D}^{z_B}})-{\cal Z}_s^{{\cal D}^{z_B}} \hat{\e}({\cal Z}_s^{{\cal D}^{z_B}})\right)ds+ e^{-\d (\tau_{z_R}\wedge t)}{\cal Q}_1({\cal Z}_{\tau_{z_R}\wedge t}^{{\cal D}^{z_B}})\right]
\end{equation}

From $z_R \le {\cal Z}_t^{{\cal D}^{z_B}} \le z_B$ for $t\in[0,\tau_{z_R})$, it is easy to show that 
	\begin{equation*}
\mathbb{E}\left[\int_0^{\tau_{z_R}} \left|e^{-\d t}\left({\cal Z}_t^{{\cal D}^{z_B}} \hat{c}({\cal Z}_t^{{\cal D}^{z_B}})-{\cal Z}_t^{{\cal D}^{z_B}} \hat{\e}({\cal Z}_t^{{\cal D}^{z_B}})\right)\right|ds\right]<\infty \;\;\;\mbox{and}\;\;\;\mathbb{E}\left[e^{-\d \tau_R}\left|{\cal Q}_1({\cal Z}_{\tau_{z_R}\wedge t}^{{\cal D}^{z_B}})\right|\right]<\infty.
	\end{equation*}
	
Since ${\cal Q}_1 (y) = -yJ_R'(y)=y{\cal X}_R(y)$ for $y\in(0,z_R]$, the {\it dominated convergence theorem} implies that 
\begin{equation}
	{\cal Q}_1(z) =\mathbb{E}\left[\int_0^{\tau_{z_R}}e^{-\d t}\left({\cal Z}_t^{{\cal D}^{z_B}} \hat{c}({\cal Z}_t^{{\cal D}^{z_B}})-{\cal Z}_t^{{\cal D}^{z_B}} \hat{\e}({\cal Z}_t^{{\cal D}^{z_B}})\right)dt+{\bf 1}_{\{\tau_{z_R}<\infty \}} e^{-\d \tau_{z_R}}{\cal Z}_{\tau_{z_R}}^{{\cal D}^{z_B}}{\cal X}_R({\cal Z}_{\tau_{z_R}}^{{\cal D}^{z_B}})\right].
\end{equation}
Since ${\cal Z}_t^{{\cal D}^{z_B}}=y e^{\d t}{\cal H}_t{\cal D}_t^{z_B}$, we have 

	\begin{align*}
-J'(y)=\mathbb{E}\left[\int_0^{\tau_R}{\cal H}_t{\cal D}_t^{z_B}\left(\hat{c}({\cal Z}_t^{{\cal D}^{z_B}})-\hat{\e}({\cal Z}_t^{{\cal D}^{z_B}})\right)dt +{\bf 1}_{\{\tau_{z_R}<\infty \}}{\cal H}_{\tau_R}{\cal D}_{\tau_R}^{z_B}{\cal X}_{R}({\cal Z}_{\tau_R}^{{\cal D}^{z_B}}) \right].
	\end{align*}

Next, we will show that 
\begin{equation*}
			-J'(y)=\mathbb{E}\left[\int_0^{\tau_R}{\cal H}_t\left(\hat{c}({\cal Z}_t^{{\cal D}^{z_B}})-\hat{\e}({\cal Z}_t^{{\cal D}^{z_B}})\right)dt +{\bf 1}_{\{\tau_{z_R}<\infty \}}{\cal H}_{\tau_R}{\cal X}_{R}({\cal Z}_{\tau_R}^{{\cal D}^{z_B}}) \right].
\end{equation*}
 For any fixed $T>0$, let us define an equivalent martingale measure $\mathbb{Q}$ by 
\begin{equation}
\dfrac{d\mathbb{Q}}{d\mathbb{P}}=e^{-\frac{1}{2}\t^2 T-\t W_T}.
\end{equation}
The Girsanov theorem implies that $W_t^{\mathbb{Q}}=W_t +\t t$ is a standard Brownian motion under the measure $\mathbb{Q}$. 

Note that 
\begin{equation}
\dfrac{d{\cal Z}_t^{{\cal D}^{z_B}}}{{\cal Z}_t^{{\cal D}^{z_B}}} =(\d-r+\t^2)dt -\t dW_t^{\mathbb{Q}} +\dfrac{d{\cal D}_t^{z_B}}{{\cal D}_t^{z_B}}
\end{equation}
under the measure $\mathbb{Q}$.

Then, similarly to the derivation above, we have 
	\begin{eqnarray}
		\begin{split}
				-J'(y) =& \mathbb{E}^{\mathbb{Q}}\left[\int_0^{\tau_{z_R}\wedge T}e^{-rt}\left(\hat{c}({\cal Z}_t^{{\cal D}^{z_B}})-\hat{\e}({\cal Z}_t^{{\cal D}^{z_B}})\right)dt + e^{-r(\tau_{z_R}\wedge T)}{\cal X}_R({\cal Z}_{\tau_{z_R}\wedge T}^{{\cal D}^{z_B}})\right]\\
				=& \mathbb{E}\left[\int_0^{\tau_{z_R}\wedge T}{\cal H}_t\left(\hat{c}({\cal Z}_t^{{\cal D}^{z_B}})-\hat{\e}({\cal Z}_t^{{\cal D}^{z_B}})\right)dt + {\cal H}_{\tau_{z_R}\wedge T}{\cal X}_R({\cal Z}_{\tau_{z_R}\wedge T}^{{\cal D}^{z_B}})\right].
		\end{split}
	\end{eqnarray}

By the {\it dominated convergence theorem}, we can obtain that 
	\begin{eqnarray}
		\begin{split}
			-J'(y)=\mathbb{E}\left[\int_0^{\tau_R}{\cal H}_t\left(\hat{c}({\cal Z}_t^{{\cal D}^{z_B}})-\hat{\e}({\cal Z}_t^{{\cal D}^{z_B}})\right)dt +{\bf 1}_{\{\tau_{z_R}<\infty \}}{\cal H}_{\tau_R}{\cal X}_{R}({\cal Z}_{\tau_R}^{{\cal D}^{z_B}}) \right].
		\end{split}
	\end{eqnarray}
\end{proof}

\begin{thm}\label{thm:main} Let $x>0$ be given. 
	\begin{itemize}
		\item[(a)] There exists a unique $y^*\in(0,z_B)$ such that 
		$$
		x = -J'(y^*).
		$$
		\item[(b)] Consider the Nash-equilibrium $(D^{*},\tau^{*})$ of $J(y^*)$ by 
		\begin{equation}
			D_t^{*} = \min\left\{1, \inf_{0\le s\le t}\dfrac{z_B}{{\cal Y}_s^*} \right\} \;\;\; t\ge0 
		\end{equation}
		and 
		\begin{equation}
			\tau^{*} = \inf\{t\ge 0 \mid {\cal Z}_t^{*} < z_R \},
		\end{equation}
		where ${\cal Z}_t^{*}={\cal Y}_t^*{\cal D}_t^*$ with ${\cal Y}_t^* = y^* e^{\d t} {\cal H}_t$. 
		
		Let  $c^*$, $\zeta^*$, and $X_{\tau^*}$ be the consumption, job state process, and wealth at time $\tau^*$ given by 	
			\begin{align*}
				c_t^* =
				\begin{cases}
					\hat{c}({\cal Z}_t^*)\;\;\;&\mbox{for}\;\; t\in[0,\tau^*), \vspace{2mm}\\
					I\left({\cal Y}_t^{R,*}\right)\;\;\;&\mbox{for}\;\; t\ge \tau^*,
				\end{cases}
			\;\;
				\zeta^* = \hat{\zeta}({\cal Z}_t^*)	=\begin{cases}
					\mathfrak{B}_1\;\;\;&\mbox{for}\;\;{\cal Z}_t^*>z_S, \vspace{2mm}\\
					\mathfrak{B}_2\;\;\;&\mbox{for}\;\;{\cal Z}_t^*\le z_S,
				\end{cases}
			\end{align*}
	and
		\begin{equation*}
			X_{\tau^*} = - J_R({\cal Z}_{\tau^*}^*)={\cal X}_R({\cal Z}_{\tau^*}^*), 
		\end{equation*}
respectively.  Here, ${\cal Y}_t^{R,*} = {\cal Z}_{\tau^*}^* e^{\d(t-\tau^*)}{\cal H}_t/{\cal H}_{\tau^*}$. Then, there exists a portfolio $\pi^*$ such that the strategy $(c^*,\pi^*, \zeta^*, \tau^*)\in{\cal A}(x)$.
		
		\item[(c)] $V(x)$ and $J(y)$ satisfy the duality relationship:
		\begin{equation}
			V(x) =\inf_{y>0}(J(y)+yx)\;\;\;\;\mbox{and}\;\;\;\;J(y)=\sup_{x>0}(V(x)-yx).
		\end{equation}
	Moreover, $(c^*,\pi^*,\zeta^*,\tau^*)$ is optimal.
	\end{itemize}
\end{thm}
\begin{proof}
\noindent{\bf  (a)} 
Since $J(y)=Q(y)$, it follows from (c) and (d) in Proposition \ref{pro:properties-Q} that (i) $J(y)$ is strictly convex in $y\in(0,z_B)$, (ii) $\lim_{y\to z_B-} J'(y)= 0$ and $\lim_{y\to0+}J'(y) =-\infty$. Thus, for given $x>0$,  there exists a unique $y^*>0$ such that 
$$
x = -J'(y^*). 
$$

\noindent{\bf  (b)}
For $y^*\in(0,z_B)$ which is a unique solution to $x=-J'(y)$, let ${\cal Y}_t^*, {\cal D}_t^*, {\cal Z}_t^*$, and $\tau^*$ be as in the statement of the theorem. 

Then, it follows from Lemma \ref{eq:equality-budet} that 
	\begin{align}
		x=&\mathbb{E}\left[\int_0^{\tau^*} {\cal H}_t(c_t^*-\e({\cal Z}_t^*))dt + {\bf 1}_{\{\tau^*<\infty \}}{\cal H}_{\tau^*}{\cal X}_R({\cal Z}_{\tau^*}^*)\right]\label{eq:equal-static1}\\
		=&\mathbb{E}\left[\int_0^{\tau^*}{\cal D}_t^* {\cal H}_t(c_t^*-\e({\cal Z}_t^*))dt + {\bf 1}_{\{\tau^*<\infty \}}{\cal D}_{\tau^*}^*{\cal H}_{\tau^*}{\cal X}_R({\cal Z}_{\tau^*}^*)\right]\label{eq:equal-static2}
	\end{align}
For given $\e>0$, define a process ${\cal D}^\e$ by 
\begin{equation}
	{\cal D}_t^\e := \dfrac{y^*{\cal D}_t^*+\e {\bf 1}_{[0,\xi)}}{y^* + \e}\;\;\;\mbox{for}\;\;t\ge 0,
\end{equation}
where $\xi$ is an arbitary stopping time belongs to ${\cal S}_{\tau^*}=\{\xi \in {\cal S} \mid 0\le \xi \le \tau^* \}$. 

Since ${\cal D}_0^\e =1$, it is clear that ${\cal D}^\e \in {\cal NI}$.  

Let us define $y^*$, ${\cal Y}_t^\e$, and ${\cal Z}_t^\e$ by 
$$
y^\e = y^* +\e,\;\;{\cal Y}_t^\e = y^\e e^{\d t} {\cal H}_t, \;\;\mbox{and}\;\;{\cal Z}_t^\e ={\cal Y}_t^\e {\cal D}_t^\e, 
$$
respectively. 

Note that 
\begin{align*}
	{\cal Z}_t^\e = 
	\begin{cases}
		{\cal Z}_t^* + \e e^{\d t} {\cal H}_t\;\;\;&\mbox{for}\;\;0\le t < \xi,\\
				{\cal Z}_t^* \;\;\;&\mbox{for}\;\;t\ge \xi.  
	\end{cases}
\end{align*}

Let $(\widehat{{\cal D}}^\e , \hat{\tau}^\e)$ be the Nash-equilibrium of $J(y^\e)$, i.e., 
\begin{equation}
	J(y^\e )= {\cal J}_0(y^\e;\widehat{{\cal D}}^\e , \hat{\tau}^\e ),
\end{equation}
where 
	\begin{align*}
		\widehat{D}_t^\e = \min\left\{1, \inf_{0\le s \le t}\dfrac{z_B}{{\cal Y}_t^\e} \right\}\;\;\;\mbox{and}\;\;\;\hat{\tau}^\e =\inf\{t\ge 0\mid {\cal Y}_t^\e \widehat{{\cal D}}_t^\e <z_R \}.
	\end{align*}
By the definition of Nash-equilibrium, 
\begin{equation*}
J_0(y^\e)+y^\e x =	{\cal J}_0 (y^\e , \widehat{{\cal D}}^\e, \hat{\tau}^\e)+y^\e x \le 	{\cal J}_0 (y^\e , {\cal D}^\e, \hat{\tau}^\e) +y^\e x.
\end{equation*}

Since $y^*$ is a unique minimizer of $J(y)+yx$, we deduce that 
\begin{eqnarray}
	\begin{split}\label{eq:a}
	{\cal J}_0(y^*; {\cal D}^*, \tau^*)+y^*x	=J(y^*) + y^* x \le J(y^\e ) +y^\e x \le 	{\cal J}_0 (y^\e , {\cal D}^\e , \hat{\tau}^\e)+y^\e x.
	\end{split}
\end{eqnarray}
Note that 
\begin{equation}\label{eq:b}
	{\cal J}_0(y^*; {\cal D}^*, \hat{\tau}^\e) +y^* x \le {\cal J}_0(y^*; {\cal D}^*, \tau^*) +y^*x, 
\end{equation}
where we have used the fact that $({\cal D}^*, \tau^*)$ is the Nash-equilibrium of $J(y^*)$. 

Define a process ${\cal N}^\e$ by 
	\begin{equation}
		{\cal N}^\e =\mathbb{E}\left[\int_0^{\hat{\tau}^\e}e^{-\d t }\left\{\left(\tilde{u}\left(\dfrac{{\cal Z}_t^*}{\k_1}\right)+\e_1 \right){\bf 1}_{\{{\cal Z}_t^\e\ge z_S \}}+\left(\tilde{u}\left(\dfrac{{\cal Z}_t^*}{\k_2}\right)+\e_2 \right){\bf 1}_{\{{\cal Z}_t^\e< z_S \}}\right\}dt + {\bf 1}_{\{\hat{\tau}^\e<\infty\}}J_R({\cal Z}_{\hat{\tau}^\e}^*) \right] +y^*x. 
	\end{equation}
By Lemma \ref{lem:z_S}, 
	\begin{align*}
		&\left(\tilde{u}\left(\dfrac{{\cal Z}_t^*}{\k_1}\right)+\e_1 \right){\bf 1}_{\{{\cal Z}_t^*\ge z_S \}}+\left(\tilde{u}\left(\dfrac{{\cal Z}_t^*}{\k_2}\right)+\e_2 \right){\bf 1}_{\{{\cal Z}_t^*< z_S \}}
		\\\ge&\left(\tilde{u}\left(\dfrac{{\cal Z}_t^*}{\k_1}\right)+\e_1 \right){\bf 1}_{\{{\cal Z}_t^\e\ge z_S \}}+\left(\tilde{u}\left(\dfrac{{\cal Z}_t^*}{\k_2}\right)+\e_2 \right){\bf 1}_{\{{\cal Z}_t^\e< z_S \}}.
	\end{align*}

This implies that 
\begin{equation}\label{eq:c}
	{\cal N}^\e \le {\cal J}_0(y^*; {\cal D}^*, \hat{\tau}^\e)+y^* x. 
\end{equation}

From \eqref{eq:a}, \eqref{eq:b}, and \eqref{eq:c}, we deduce that 
\begin{equation}
	{\cal N}^\e \le {\cal J}_0(y^\e ; {\cal D}^\e , \hat{\tau}^\e) +y^\e x.
\end{equation}

Hence, 
\begin{eqnarray}
	\begin{split}\label{eq:inquality-stopping}
				0\le &\mathbb{E}\left[\int_0^{\xi}e^{-\d t}\left\{\left(\dfrac{\tilde{u}(\frac{{\cal Z}_t^\e}{\k_1})-\tilde{u}(\frac{{\cal Z}_t^*}{\k_1})}{\e}+\e_1 e^{\d t}{\cal H}_t\right){\bf 1}_{\{{\cal Z}_t^\e \ge z_S \}} +\left(\dfrac{\tilde{u}(\frac{{\cal Z}_t^\e}{\k_2})-\tilde{u}(\frac{{\cal Z}_t^*}{\k_2})}{\e}+\e_2 e^{\d t}{\cal H}_t\right){\bf 1}_{\{{\cal Z}_t^\e < z_S \}}\right\} dt\right.\\ +&\left. {\bf 1}_{\{\hat{\tau}^\e \le \xi <\infty \}}\left(\dfrac{J_R({\cal Z}_{\hat{\tau}^\e-}^\e)-J_R({\cal Z}_{\hat{\tau}^\e}^*)}{\e}\right)\right]+x.
	\end{split}
\end{eqnarray}
By applying the {\it dominated convergence theorem}, it is easy to confirm that 
	\begin{align}
	\label{eq:dominated}
	\lim_{\e \downarrow 0} \mathbb{E}\left[\int_0^{\xi}{\cal H}_t\left(\e_1{\bf 1}_{\{{\cal Z}_t^\e \ge z_S \}}+\e_2{\bf 1}_{\{{\cal Z}_t^\e < z_S \}}\right) dt \right]=&\mathbb{E}\left[\int_0^{\xi}{\cal H}_t\left(\e_1{\bf 1}_{\{{\cal Z}_t^* \ge z_S \}}+\e_2{\bf 1}_{\{{\cal Z}_t^* < z_S \}}\right) dt \right]\\=&\mathbb{E}\left[\int_0^{\xi} {\cal H}_t\hat{e}({\cal Z}_t^*)dt \right]. \nonumber
	\end{align}
Note that 
$$
\tilde{u}(\frac{{\cal Z}_t^\e}{\k_1})-\tilde{u}(\frac{{\cal Z}_t^*}{\k_1})\le 0, \;\;\tilde{u}(\frac{{\cal Z}_t^\e}{\k_2})-\tilde{u}(\frac{{\cal Z}_t^*}{\k_2})\le 0,\;\;\mbox{and}\;\;J_R({\cal Z}_{\hat{\tau}^\e-}^\e)-J_R({\cal Z}_{\hat{\tau}^\e}^*)\le 0.
$$

Thus, it follows from Fatou's lemma and \eqref{eq:dominated} in \eqref{eq:inquality-stopping} that 
	\begin{align*}
	0\le -\mathbb{E}\left[\int_0^{\xi}{\cal H}_t (c_t^*-\hat{\e}({\cal Z}_t^*))dt +{\bf 1}_{\{\xi=\tau^*<\infty \}}{\cal H}_{\tau^*} {\cal X}_R({\cal Z}_{\tau^*}^*)\right]+x, 
	\end{align*}
where we have used tha fact that $\lim_{\e \downarrow 0}\hat{\tau}^\e =\tau^*$ and $0\le \xi \le \tau^*$.

Therefore, we get 
\begin{equation}\label{eq:inequality_static}
	\mathbb{E}\left[\int_0^{\xi}{\cal H}_t (c_t^*-\hat{\e}({\cal Z}_t^*))dt +{\bf 1}_{\{\xi=\tau^*<\infty \}}{\cal H}_{\tau^*} {\cal X}_R({\cal Z}_{\tau^*}^*)\right]\le x. 
\end{equation}
Since  inequality \eqref{eq:inequality_static} holds for any stopping time $\xi\in{\cal S}_{\tau^*}$, 
\begin{equation}
\sup_{\xi\in {\cal S}_{\tau^*}}\mathbb{E}\left[\int_0^{\xi}{\cal H}_t (c_t^*-\hat{\e}({\cal Z}_t^*))dt +{\bf 1}_{\{\xi=\tau^*<\infty \}}{\cal H}_{\tau^*} {\cal X}_R({\cal Z}_{\tau^*}^*)\right]\le x
\end{equation}
By Proposition \ref{pro:static-budget} (b), there exists a portfolio $\pi_t^*$ such that 
$$
(c^*, \pi^*, \zeta^*, \tau^*)\in{\cal A}(x)
$$
with $X_{\tau^*} = {\cal X}_R({\cal Z}_{\tau^*}^*)$\\. 
 
\noindent{\bf  (c)} By Lemma \ref{eq:equality-budet}, we have 
\begin{align*}
		x=-J'(y^*)=&\mathbb{E}\left[\int_0^{\tau^*}{\cal H}_t\left(c_t^*-\hat{\e}({\cal Z}_t^{*})\right)dt +{\bf 1}_{\{\tau^{*}<\infty \}}{\cal H}_{\tau^*}{\cal X}_{R}({\cal Z}_{\tau^*}^{*}) \right]\\
	=&\mathbb{E}\left[\int_0^{\tau^*}{\cal H}_t{\cal D}_t^{*}\left(c_t^*-\hat{\e}({\cal Z}_t^{*})\right)dt +{\bf 1}_{\{\tau^{*}<\infty \}}{\cal H}_{\tau^*}{\cal D}_{\tau^*}^{*}{\cal X}_{R}({\cal Z}_{\tau^*}^{*}) \right].
\end{align*}

Note that 
	\begin{align*}
		x=&\mathbb{E}\left[\int_0^{\tau^*}{\cal H}_t\left(c_t^*-\hat{\e}({\cal Z}_t^{*})\right)dt +{\bf 1}_{\{\tau^{*}<\infty \}}{\cal H}_{\tau^*}{\cal X}_{R}({\cal Z}_{\tau^*}^{*}) \right]\\
		\le&\sup_{\xi\in {\cal S}_{\tau^*}}\mathbb{E}\left[\int_0^{\xi}{\cal H}_t (c_t^*-\hat{\e}({\cal Z}_t^*))dt +{\bf 1}_{\{\xi=\tau^*<\infty \}}{\cal H}_{\tau^*} {\cal X}_R({\cal Z}_{\tau^*}^*)\right]\le x.
	\end{align*}
That is,
	\begin{equation*}
		x=\sup_{\xi\in {\cal S}_{\tau^*}}\mathbb{E}\left[\int_0^{\xi}{\cal H}_t (c_t^*-\hat{\e}({\cal Z}_t^*))dt +{\bf 1}_{\{\xi=\tau^*<\infty \}}{\cal H}_{\tau^*} {\cal X}_R({\cal Z}_{\tau^*}^*)\right].
	\end{equation*}

On the other hand, 
	\begin{align*}
		xy^*=-y^* J'(y^*) =&\mathbb{E}\left[\int_0^{\tau^*}y^*{\cal H}_t{\cal D}_t^{*}\left(c_t^*-\hat{\e}({\cal Z}_t^{*})\right)dt +{\bf 1}_{\{\tau^{*}<\infty \}}y^*{\cal H}_{\tau^*}{\cal D}_{\tau^*}^{*}{\cal X}_{R}({\cal Z}_{\tau^*}^{*}) \right]\\
		=&\mathbb{E}\left[\int_0^{\tau^*}e^{-\d t}{\cal Z}_t^{*}\left(c_t^*-\hat{\e}({\cal Z}_t^{*})\right)dt +{\bf 1}_{\{\tau^{*}<\infty \}}e^{-\d \tau^*}{\cal Z}_{\tau^*}^{*}{\cal X}_R({\cal Z}_{\tau^*}^{*})\right]\\
		=&\mathbb{E}\left[\int_0^{\tau^*}e^{-\d t}\left(u(\k_1 c_t^*){\bf 1}_{\{\zeta_t^*=\mathfrak{B}_1 \}}+u(\k_2 c_t^*){\bf 1}_{\{\zeta_t^*=\mathfrak{B}_2 \}}\right)dt + {{\bf 1}_{\{\tau^*<\infty \}}}\int_{\tau^*}^\infty e^{-\d t}u(c_t^{*})dt \right]-J(y^*).
	\end{align*}
Since $J(y^*)+y^*x \ge \inf_{y>0}\left(J(y)+yx\right)$, the weak duality in \eqref{eq:weak-duality2} implies that  
	\begin{align*}
		&\mathbb{E}\left[\int_0^{\tau^*}e^{-\d t}\left(u(\k_1 c_t^*){\bf 1}_{\{\zeta_t^*=\mathfrak{B}_1 \}}+u(\k_2 c_t^*){\bf 1}_{\{\zeta_t^*=\mathfrak{B}_2 \}}\right)dt + {{\bf 1}_{\{\tau^*<\infty \}}}\int_{\tau^*}^\infty e^{-\d t}u(c_t^{*})dt \right]\\
		\ge &\sup_{(c,\pi, \d, \tau)\in \mathcal{A}(x)}\mathbb{E}\left[\int_0^{\tau}e^{-\d t}\left(u(\k_1 c_t){\bf 1}_{\{\zeta_t=\mathfrak{B}_1 \}}+u(\k_2 c_t){\bf 1}_{\{\zeta_t=\mathfrak{B}_2 \}}\right)dt +{{\bf 1}_{\{\tau<\infty \}}} \int_{\tau}^\infty e^{-\d t}u(c_t)dt \right].
	\end{align*}
Therefore, it follows from $(c^*, \pi^*, \zeta^*, \tau^*)\in{\cal A}(x)$ that 
\begin{equation*}
	V(x)= \mathbb{E}\left[\int_0^{\tau^*}e^{-\d t}\left(u(\k_1 c_t^*){\bf 1}_{\{\zeta_t^*=\mathfrak{B}_1 \}}+u(\k_2 c_t^*){\bf 1}_{\{\zeta_t^*=\mathfrak{B}_2 \}}\right)dt + {{\bf 1}_{\{\tau^*<\infty \}}}\int_{\tau^*}^\infty e^{-\d t}u(c_t^{*})dt \right].
\end{equation*}
That is, 
$$
V(x)=\inf_{y>0}\left(J(y)+yx\right),\;\;J(y)=\sup_{x>0}\left(V(x)-yx\right)
$$
and
$(c^*, \pi^*, \zeta^*, \tau^*)\in{\cal A}(x)$ is {\it optimal}. 
\end{proof}

\begin{rem}
	The agent stays at $\mathfrak{B}_2$ or switches from $\mathfrak{B}_1$  to $\mathfrak{B}_2,$ a job with higher satisfaction before retirement if $z_R<z_S<z_B$. If $z_S\le z_R$ or $z_S\ge z_B$ the agent stays always at $\mathfrak{B}_1$ or $\mathfrak{B}_2$, respectively, before retirement.
\end{rem}

\section{Conclusion}\label{sec:conclusion}

We have studied the optimization problem of the job choice, retirement, consumption and portfolio selection of a borrowing constrained agent. We have derived a solution in concrete form by transforming the problem into a dual two-person zero-sum game. In this paper,
we assume a constant investment opportunity, and consideration of  a general market environment would be an interesting topic for future research.  

\bibliographystyle{apalike}
\bibliography{library}

\begin{thebibliography}{}

\bibitem[Bayraktar and Young, 2011]{BY11}
Bayraktar, E. and Young, V.~R. (2011).
\newblock {Proving regularity of the minimal probability of ruin via a game of
  stopping and control}.
\newblock {\em Finance Stoch.}, 15:785–818.

\bibitem[Choi et~al., 2007]{CSS2007}
Choi, K., Shim, G., and Shin, Y. (2007).
\newblock { Optimal Portfolio, Consumption-Leisure and Retirement Choice
  Problem with CES Utility}.
\newblock {\em Mathematical Finance}, 18(3):445--472.

\bibitem[Cox and Huang, 1989]{CoxH}
Cox, J. and Huang, C. (1989).
\newblock {Optimal Consumption and Portfolio Polices when Asset Prices Follow a
  Diffusion Process}.
\newblock {\em Journal of Economic Theory}, 49(1):33--83.

\bibitem[Dixit and Pindyck, 1994]{DixitPindyck}
Dixit, A. and Pindyck, R. (1994).
\newblock {\em Investment under Uncertainty}.
\newblock Princeton University Press.

\bibitem[Domeij and Flod\'{e}n, 2006]{DF2006}
Domeij, D. and Flod\'{e}n, M. (2006).
\newblock { The labor-supply elasticity and borrowing constraints: Why
  estimates are biased}.
\newblock {\em Review of Economic Dynamics}, 9(4):242--262.

\bibitem[Dumas, 1989]{Dumas1989}
Dumas, B. (1989).
\newblock {Two-person dynamic equilibrium in the capital market}.
\newblock {\em Review of Financial Studies}, 2:157--188.

\bibitem[Dybvig and Liu, 2010]{DL2010}
Dybvig, P. and Liu, H. (2010).
\newblock {Lifetime consumption and investment: Retirement and constrained
  Borrowing}.
\newblock {\em Journal of Economic Theory}, 145(3):885--907.

\bibitem[Dybvig and Liu, 2011]{DL2011}
Dybvig, P. and Liu, H. (2011).
\newblock {Verification Theorems for Models of Optimal Consumption and
  Investment}.
\newblock {\em Mathematics of Operations Research}, 36(4):620--635.

\bibitem[El\;Karoui and Jeanblanc-Picqu\'{e}, 1998]{KarouiJeanblanc1998}
El\;Karoui, N. and Jeanblanc-Picqu\'{e}, M. (1998).
\newblock {Optimization of Consumption with Labor Income}.
\newblock {\em Finance and Stochastic}, 2(4):409--440.

\bibitem[Farhi and Panageas, 2007]{FP2007}
Farhi, E. and Panageas, S. (2007).
\newblock {Saving and investing for early retirement: A theoretical analysis}.
\newblock {\em Journal of Financial Economics}, 83(1):87--121.

\bibitem[Grossman and Laroque, 1990]{GL}
Grossman, S. and Laroque, G. (1990).
\newblock {Asset Pricing and Optimal Portfolio Choice in the Presence of
  Illiquid Durable Consumption Goods}.
\newblock {\em Econometrica}, 58:25--51.

\bibitem[Hamad\'{e}ne, 2006]{H06}
Hamad\'{e}ne, S. (2006).
\newblock {Mixed zero-sum stochastic differential game and American game
  options}.
\newblock {\em SIAM J. Control Optim.}, 45:496–518.

\bibitem[Harrison, 1985]{Harrison}
Harrison, M. (1985).
\newblock {\em { Brownian Motion and Stochastic Flow Systems}}.
\newblock Wiley.

\bibitem[He and Pag\'es, 1993]{HP}
He, H. and Pag\'es, H. (1993).
\newblock {Labor Income, Borrowing Constraints, and Equilibrium Asset Prices}.
\newblock {\em Economic Theory}, 3(4):663--696.

\bibitem[Hern\'andez-Hern\'andez et~al., 2015]{HSZ15}
Hern\'andez-Hern\'andez, D., Simon, R., and Zervos, M. (2015).
\newblock {A zero-sum game between a singular stochastic controller and a
  discretionary stopper}.
\newblock {\em Ann. Appl. Probab}, 25(1):46--80.

\bibitem[Hern\'andez-Hern\'andez and Yamazaki, 2015]{HY15}
Hern\'andez-Hern\'andez, D. and Yamazaki, K. (2015).
\newblock {Games of singular control and stopping driven by spectrally
  one-sided L\'evy processes}.
\newblock {\em Stochastic Processes and their Applications}, 125:1–38.

\bibitem[Karatzas et~al., 1987]{KLS}
Karatzas, I., Lehoczky, J., and Shreve, S. (1987).
\newblock {Optimal Portfolio and Consumption Decisions for a "Small Investor"
  on a Finite Horizon}.
\newblock {\em SIAM Journal on Control and Optimization}, 25(6):1557--1586.

\bibitem[Karatzas and Shreve, 1991]{KS2}
Karatzas, I. and Shreve, S. (1991).
\newblock {\em Brownian Motion and Stochastic Calculus}.
\newblock Springer-Verlag New York.

\bibitem[Karatzas and Shreve, 1998]{KS}
Karatzas, I. and Shreve, S. (1998).
\newblock {\em Methods of Mathematical Finance}.
\newblock Springer-Verlag New York.

\bibitem[Karatzas and Wang, 2000]{KW2000}
Karatzas, I. and Wang, H. (2000).
\newblock {Utility Maximization with Discretionary Stopping}.
\newblock {\em SIAM Journal on Control and Optimization}, 39(1):306--329.

\bibitem[Karatzas and Zamfirescu, 2008]{KZ08}
Karatzas, I. and Zamfirescu, I.-M. (2008).
\newblock {Martingale approach to stochastic differential games of control and
  stopping}.
\newblock {\em Ann. Probab.}, 36:1495–1527.

\bibitem[Knudsen et~al., 1998]{KMZ}
Knudsen, T., Meister, B., and Zervos, M. (1998).
\newblock {Valuation of investments in real assets with implications for the
  stock prices}.
\newblock {\em SIAM J. Control and Optim.}, 36:2082--2102.

\bibitem[Lee et~al., 2019]{Lee-et-al-2019}
Lee, H., Shim, G., and Shin, Y. (2019).
\newblock {Borrowing Constraints, Effective Flexibility in Labor Supply, and
  Portfolio Selection}.
\newblock {\em Mathematics and Financial Economics}, 13(2):173--208.

\bibitem[Lim and Shin, 2011]{LS2011}
Lim, B. and Shin, Y. (2011).
\newblock {Optimal investment, consumption and retirement decision with
  disutility and borrowing constraints}.
\newblock {\em Quantitative Finance}, 11(10):1581--1592.

\bibitem[Maitra and Sudderth, 1996]{MS96}
Maitra, A. and Sudderth, W.~D. (1996).
\newblock {\em { The gambler and the stopper. In Statistics, Probability and
  Game Theory (T. S. Ferguson, L. S. Shapley and J. B. MacQueen, eds.)}},
  volume~30.
\newblock IMS, Hayward, CA.

\bibitem[Rendon, 2006]{Rendon2006}
Rendon, S. (2006).
\newblock {Job Search and Asset Accumulation under Borrowing Constraints}.
\newblock {\em International Economic Review}, 47(2):33--263.

\bibitem[Shim et~al., 2018]{SKS}
Shim, G., Koo, J., and Shin, Y. (2018).
\newblock {Reversible Job-Switching Opportunities and Portfolio Selection}.
\newblock {\em Applied Mathematics and Optimization}, 77:197--228.

\bibitem[Weerasinghe, 2006]{W06}
Weerasinghe, A. (2006).
\newblock {controller and a stopper game with degenerate variance control}.
\newblock {\em Electron. Commun. Probab}, 11:89–99.

\bibitem[Yang and Koo, 2018]{YK}
Yang, Z. and Koo, H. (2018).
\newblock {Optimal Consumption and Portfolio Selection with Early Retirement
  Option}.
\newblock {\em Mathematics of Operations Research}, 43(4):1378--1404.

\end{thebibliography}
\appendix

\section{Properties of $J_R(y)$}\label{sec:after-dual}

Proposition 4.1 in \citet{KMZ} is useful to get the analytic properties of $J_R(y)$ as well as the closed form solution of $J_R(y)$. Thus, we provide a representation of the proposition in our notation:
\begin{pro}[Proposition 4.1 in \citet{KMZ}]\label{pro:review:KMZ}
	Let $g(y)$ be an arbitrary measurable function defined on $(0,\infty)$. Then the following conditions are equivalent:
	\begin{itemize}
		\item[(i)] for every $y>0$ 
		$$
		\mathbb{E}\left[\int_0^\infty e^{-\d t}|g({\cal Y}_t)|dt \right]<\infty,
		$$
		\item[(ii)] for every $y>0$ 
		$$
		\int_0^y \eta^{-n_2-1} |g(\eta)|d\eta +\int_y^\infty \eta^{-n_1-1}|g(\eta)|d\eta <\infty. 
		$$
	\end{itemize}
	Let us denote $\Xi_g(y)$ by 
	$$
	\Xi_g(y) = \mathbb{E}\left[\int_0^\infty e^{-\d t} g({\cal Y}_t)dt\right]. 
	$$
	Under the condition (i) or (ii), the following statements are true:
	\begin{itemize}
		\item[(a)] $\liminf_{y\downarrow 0}y^{-n_2}|g(y)|=\liminf_{y\uparrow \infty}y^{-n_1}|g(y)|=0$,
		\item[(b)] $\Xi_g$ has a following form: 
			\begin{align*}
			\Xi_g (y) = \dfrac{2}{\t^2(n_1-n_2)}\left[y^{n_2}\int_0^y \eta^{-n_2-1}g(\eta)d\eta +y^{n_1}\int_y^\infty \eta^{-n_1-1}g(\eta)d\eta\right],
			\end{align*}
		\item[(c)] $\Xi_g$ is twice differentiable and 
		$$
		\dfrac{\t^2}{2}y^2\Xi_g''(y) + (\d-r)y\Xi_g'(y) -\d \Xi_g(y) +g(y)=0,
		$$
		\item[(d)] there exists a positive constant $C$ such that 
		$$
		|\Xi_g'(y)|\le C(y^{n_1-1}+y^{n_2-1})\;\;\;\mbox{for all}\;\;y>0,
		$$
		\item[(e)] $\lim_{t\to \infty} e^{-\d t} \mathbb{E}\left[|\Xi_g({\cal Y}_t)|\right]=0$. 
	\end{itemize}
\end{pro}

From Assumption \ref{as:utility2}, 
	\begin{align*}
	\mathbb{E}\left[\int_0^\infty {\cal H}_t I({\cal Y}_t)dt\right] =\dfrac{1}{y}\mathbb{E}\left[\int_0^\infty e^{-\d t} {\cal Y}_t I({\cal Y}_t)dt \right]<\infty. 
	\end{align*}

By Proposition \ref{pro:review:KMZ}, we have 
\begin{equation}\label{eq:c-hat-integral}
\int_0^y \eta^{-n_2}I(\eta)d\eta +\int_y^\infty \eta^{-n_1}I(\eta)d\eta<\infty.
\end{equation}

Let us denote ${\cal X}_R (y)$ by 
$$
{\cal X}_R (y) =\mathbb{E}\left[\int_0^\infty {\cal H}_t I({\cal Y}_t)dt\right]=\dfrac{1}{y}\mathbb{E}\left[\int_0^\infty e^{-\d t} {\cal Y}_t I({\cal Y}_t)dt \right]. 
$$

Since 
$$
{\cal X}_R (y) =\mathbb{E}\left[\int_0^\infty {\cal H}_t I(ye^{\d t}{\cal H}_t)dt\right],
$$

the monotone convergence theorem and the dominated convergence theorem imply that 
\begin{align}\label{eq:XR-inf}   
\lim_{y\downarrow 0}{\cal X}_R (y) =\mathbb{E}\left[\int_0^\infty \lim_{y\downarrow 0} {\cal H}_t I(ye^{\d t}{\cal H}_t)dt\right]=\infty
\end{align}
and 
\begin{align}
\lim_{y\uparrow \infty}{\cal X}_R (y) =\mathbb{E}\left[\int_0^\infty \lim_{y\uparrow \infty} {\cal H}_t I(ye^{\d t}{\cal H}_t)dt\right]=0.
\end{align}

By Proposition \ref{pro:review:KMZ}, we have 
	\begin{align*}
	{\cal X}_R(y) = \dfrac{2}{\t^2(n_1-n_2)}\left[y^{n_2-1}\int_0^y \eta^{-n_2}I(\eta)d\eta +y^{n_1-1}\int_y^\infty \eta^{-n_1}I(\eta)d\eta\right].
	\end{align*}
and 
$$
\liminf_{y\downarrow 0}y^{-n_2+1}I(y)=\liminf_{y\uparrow \infty}y^{-n_1+1}I(y)=0.
$$
By utilizing the integration by parts for Riemann-Stieltjes integral, we have 
	\begin{align*}
	{\cal X}_R'(y)=& \dfrac{2}{\t^2(n_1-n_2)}\left[(n_2-1)y^{n_2-2}\int_0^y \eta^{-n_2}I(\eta)d\eta +(n_1-1)y^{n_1-2}\int_y^\infty \eta^{-n_1}I(\eta)d\eta\right]\\
	=&\dfrac{2}{\t^2(n_1-n_2)}\left[y^{n_2-2}\liminf_{\eta\downarrow 0}(\eta^{1-n_2}I(\eta))-y^{n_1-2}\liminf_{\eta\uparrow \infty}(\eta^{1-n_1}I(\eta))+y^{n_2-2}\int_0^y \eta^{1-n_2}dI(\eta)\right.\\ +&\left.y^{n_1-2}\int_y^\infty \eta^{1-n_1}dI(\eta)\right].
	\end{align*}

Since $I(\eta)$ is strictly decreasing in $0<\eta\le u'(0)$, we deduce that ${\cal X}_R(y)$ is strictly decreasing in $y>0$, i.e., 
\begin{equation}\label{eq:sign-X-prime}
{\cal X}_R'(y)<0. 
\end{equation}

Note that for any $y>0$ 
\begin{equation*}
u(I(y))-yI(y)=\tilde{u}(y)\;\;\mbox{and}\;\;\tilde{u}'(y)=-I(y),
\end{equation*}
and thus
\begin{equation*}
u(I(\eta))=u(I(y))-yI(y)+\eta I(\eta)+\int_{\eta}^y I(\eta)d\eta. 
\end{equation*}

Since $\mathbb{E}\left[\int_0^\infty e^{-\d t} {\cal Y}_t I({\cal Y}_t)dt \right]<\infty$, it follows from Proposition \ref{pro:review:KMZ} that 
	\begin{eqnarray}
	\begin{split}\label{eq:u-integral}
	&\int_0^y \eta^{-n_2-1}|u(I(\eta))|d\eta +\int_y^\infty \eta^{-n_1-1}|u(I(\eta))|d\eta\\
	\le-&\dfrac{1}{n_2}y^{-n_2}|u(I(y))-yI(y)|+\int_0^y \eta^{-n_2}I(\eta)d\eta +\int_0^y \int_\eta^y \eta^{-n_2-1}I(\eta)d\eta d\eta\\ +&\dfrac{1}{n_1}y^{-n_1}|u(I(y))-yI(y)|+\int_y^\infty \eta^{-n_1}I(\eta)d\eta +\int_y^\infty \int_y^\eta \eta^{-n_1-1}I(\eta)d\eta d\eta\\
	=-&\dfrac{1}{n_2}y^{-n_2}|u(I(y))-yI(y)|+\dfrac{1}{n_1}y^{-n_1}|u(I(y))-yI(y)|\\&+\left(1+\dfrac{1}{n_1}\right)\int_y^\infty \eta^{-n_1}I(\eta)d\eta +\left(1-\dfrac{1}{n_2}\right)\int_0^y \eta^{-n_2}I(\eta)d\eta <\infty. 
	\end{split}
	\end{eqnarray}
where we have used Fubini's theorem in last equality.

From \eqref{eq:c-hat-integral} and \eqref{eq:u-integral}, 
\begin{equation}\label{eq:utilde-integral}
\int_0^y \eta^{-n_2-1}|\tilde{u}(\eta)|d\eta +\int_y^\infty \eta^{-n_1-1}|\tilde{u}(\eta)|d\eta<\infty.
\end{equation}

By Proposition \ref{pro:review:KMZ}, we deduce that 
	\begin{equation}
	J_R(y)= \dfrac{2}{\t^2(n_1-n_2)}\left[y^{n_2}\int_0^y \eta^{-n_2-1}\tilde{u}(\eta)d\eta +y^{n_1}\int_y^\infty \eta^{-n_1-1}\tilde{u}(\eta)d\eta\right]
	\end{equation}
and 
\begin{equation}\label{eq:J_R-ODE}
\dfrac{\t^2}{2}y^2 J_R''(y)+(\d-r)yJ_R'(y)-\d J_R + \tilde{u}(y)=0. 
\end{equation}

Moreover, it follows from \eqref{eq:utilde-integral} that 
$$
\liminf_{y\downarrow 0}y^{-n_2}|\tilde{u}(y)|=\liminf_{y\uparrow \infty}y^{-n_1}|\tilde{u}(y)|=0.
$$
Therefore, the integration by parts implies that 
	\begin{eqnarray}
	\begin{split}\label{eq:J0-prime}
	J_R'(y)=&\dfrac{2}{\t^2(n_1-n_2)}\left[n_2y^{n_2-1}\int_0^y \eta^{-n_2-1}\tilde{u}(\eta)d\eta +n_1y^{n_1-1}\int_y^\infty \eta^{-n_1-1}\tilde{u}(\eta)d\eta\right]\\
	=&-\dfrac{2}{\t^2(n_1-n_2)}\left[y^{n_2-1}\int_0^y \eta^{-n_2}I(\eta)d\eta +y^{n_1-1}\int_y^\infty \eta^{-n_1}I(\eta)d\eta\right]=-{\cal X}_R(y).
	\end{split}
	\end{eqnarray}

From \eqref{eq:sign-X-prime}, we have $J_R''(y) = -{\cal X}_R'(y)>0$. That is, $J_R(y)$ is strictly convex in $y>0$.


\end{document}